\documentclass[10pt,letterpaper]{amsart}
\usepackage[letterpaper,margin=0.725in]{geometry}

\usepackage{graphicx}
\usepackage[colorlinks, citecolor=blue, linkcolor=black]{hyperref}

\usepackage{epsfig,color}
 \usepackage{hyperref}
 \usepackage{float}
 
\usepackage{mathrsfs}

  \usepackage{setspace}
 
\makeatother
\theoremstyle{definition}
\newtheorem{theorem}{Theorem} 
 
\newtheorem{problem}{Open Problem}

\newtheorem{remark}{Remark}[section] 
\newtheorem{proposition}{Proposition} 

\title[\textsf{\textbf{A survey of closed self-shrinkers with symmetry}}]{A survey of closed self-shrinkers with symmetry}
\author[]{Gregory Drugan}  
\author[]{Hojoo Lee}  
\author[]{Xuan Hien Nguyen}

 \numberwithin{equation}{section}
 
\begin{document}

 \begin{abstract}
In this paper, we survey known results on closed self-shrinkers for mean curvature flow and discuss techniques used in recent constructions of closed self-shrinkers with classical rotational symmetry. We also propose new existence and uniqueness problems for closed self-shrinkers with bi-rotational symmetry and provide numerical evidence for the existence of new examples.
 \end{abstract}

 \maketitle

 

 \section{Introduction}

The self-shrinking solitons for mean curvature flow are ancient solutions to the flow that evolve by ``shrinking'' self-similarly about a point. A time-slice for a self-shrinking flow is a hypersurface, called a \emph{self-shrinker}, that satisfies a non-linear second order elliptic equation involving the mean curvature. When the flow shrinks about the origin, the self-shrinker equation for a time-slice $\Sigma$ is $${\Delta}_{ {\mathbf{g}}_{{}_{\Sigma}}  }  \mathbf{X} = \alpha {\mathbf{X}}^{\bot},$$ where $\alpha < 0$ is a constant and ${\Delta}_{ {\mathbf{g}}_{{}_{\Sigma}}  }  \mathbf{X}$ equals to the mean curvature vector for $\Sigma$.

The first variation formula shows that self-shrinkers are minimal hypersurfaces in the Riemmanian manifold $\mathbb{R}^{n+1}$ with conformal metric $$e^{ \frac{ \alpha \vert   \mathbf{X}  \vert^{2} }{n} } \left( {dx_{1}}^{2}+\cdots+{dx_{n+1}}^{2} \right).$$ This geometric variational characterization leads to the reduction that self-shrinkers with rotational or bi-rotational symmetry correspond to geodesics in the plane equipped with induced conformal metrics. So, the existence and uniqueness of closed self-shrinkers with one of these types of symmetry can be reduced to the study of closed geodesics in two dimensional manifolds.

Self-shrinkers play a vital role in the theory of mean curvature flow and admit a number of interesting applications. Husiken's monotonicity formula~\cite[Section 3]{Huisken 1990} shows that self-shrinkers model the asymptotic behavior of mean curvature flow at type I singularities. Self-shrinkers can also be used as barriers to explore different phenomena for solutions to mean curvature flow. For instance, the existence of a self-shrinking torus was recently used in \cite{Drugan Nguyen 2016} to construct initial entire graphs whose mean curvature flow evolves away from the heat flow.

The focus of this survey is on closed self-shrinkers with rotational or bi-rotational symmetry. The paper is organized as follows: 

\begin{itemize}

\item In Section \ref{Section 2}, we introduce the self-shrinker equation and highlight various elliptic and parabolic characterizations of self-shrinkers.  

\item In Section \ref{Section 3}, we discuss existence and uniqueness results for closed self-shrinkers. We begin with the classification of self-shrinking curves in ${\mathbb{R}}^{2}$, including proofs of two geometric conservation laws for self-shrinking curves~\cite{Abresch Langer 1986, Epstein Weinstein 1987, Halldorsson 2012}. Next, we review rigidity results for closed self-shrinkers due to Huisken~\cite{Huisken 1984, Huisken 1990} and Brendle~\cite{Brendle 2016}. Finally, we mention several examples of self-shrinkers with symmetry \cite{Angenent 1992, Drugan 2015, Drugan Kleene 2017, Moller 2011}.

\item In Section \ref{Section 4}, we give a detailed sketch of the recent variational proof for the existence of an embedded, torus self-shrinker with rotational symmetry \cite{Drugan Nguyen preprint}. The proof uses a modified curve shortening flow to find a closed geodesic in the upper half-plane. A new feature of this variational proof is that it comes with an upper bound for the weighted length of the constructed geodesic.

\item In Section \ref{Section 5}, we expand on how the shooting method can be used to construct closed self-shrinkers with rotational symmetry. In the first part of this section, we outline the analysis used in~\cite{Drugan 2015} to construct an immersed sphere self-shrinker with rotational symmetry. Then, we illustrate how the behavior of geodesics for three shooting problems can be used to generate more examples of closed self-shrinkers with rotational symmetry~\cite{Drugan Kleene 2017}. We end this section with a discussion on the role of continuity in the shooting method.

\item In Section \ref{Section 6}, we consider the problem of constructing closed self-shrinkers with bi-rotational symmetry. Here we propose the existence of various bi-rotational ${\mathbb{T}}^{3}$ and ${\mathbb{S}}^{3}$ self-shrinkers in ${\mathbb{R}}^{4}$ and present numerical approximations of their symmetric profile curves. 

\item Finally, in Section \ref{Section 7}, we present a list of old and new open problems on the existence and uniqueness of closed self-shrinkers. 

\end{itemize}



\section{Elliptic and parabolic characterizations of self-shrinkers} \label{Section 2} 

\subsection{Self-shrinker equation}

A hypersurface $\Sigma$ in Euclidean space ${\mathbb{R}}^{n+1}$ is a \emph{self-shrinker} with the constant coefficient $\alpha<0$ when it solves the quasi-linear elliptic partial differential system of second order
 \begin{equation} \label{SS_eqn1}
{\Delta}_{ {\mathbf{g}}_{{}_{\Sigma}}  }  \mathbf{X} = \alpha {\mathbf{X}}^{\bot},
 \end{equation}
where $\mathbf{X}$ is the position vector for $\Sigma$, ${\mathbf{g}}_{{}_{\Sigma}}$ is the induced metric on $\Sigma$, ${\Delta}_{ {\mathbf{g}}_{{}_{\Sigma}} }$ is the Laplace-Beltrami operator for ${\mathbf{g}}_{{}_{\Sigma}}$, and $\mathbf{X}^{\bot}$ is the orthogonal projection of $\mathbf{X}$ into the normal bundle of $\Sigma$. We note that ${\Delta}_{ {\mathbf{g}}_{{}_{\Sigma}}  }  \mathbf{X}$ equals the mean curvature vector $\mathbf{H}$ for $\Sigma$. When we orient $\Sigma$ by a smooth unit normal vector field $\mathbf{N}$ and introduce the mean curvature $H = \mathbf{H} \cdot \mathbf{N} = \left( {\Delta}_{ {\mathbf{g}}_{{}_{\Sigma}}  }  \mathbf{X} \right)  \cdot \mathbf{N}$, we obtain the 
 the scalar partial differential equation
 \begin{equation} \label{SS_eqn2}
 H = \alpha \, {\mathbf{X}} \cdot  \mathbf{N}.
 \end{equation}
Though (\ref{SS_eqn2}) resembles the classical constant mean curvature equation, in general, the standard techniques (such as method of moving planes) do not directly work for self-shrinkers. 

Locally, a self-shrinker may be written as the graph of a function $u: \Omega \subset \mathbb{R}^n \to \mathbb{R}$, where $u$ is a solution to 
 \begin{equation} \label{SS_eqn_local1}
\textrm{div}_{\mathbb{R}^n} \left( \frac{Du}{\sqrt{1 + |Du|^2} } \right) = \alpha \frac{u - x \cdot Du}{\sqrt{1 + |Du|^2} }.
 \end{equation}

\subsection{Ancient solutions for mean curvature flow} 

A self-shrinker $\Sigma$ corresponds to the $t=\left(t_{0} + \frac{1}{2\alpha}\right)$ time slice of a mean curvature flow that shrinks to the origin at the extinction time $t_{0}$. More explicitly, the one parameter family of hypersurfaces $${\Sigma}_{t}:=\sqrt{2 \alpha \left( t - t_{0}  \right) \,} \, {\Sigma}$$ is a solution to mean curvature flow (MCF) 
\begin{equation} \label{MCF_eqn}
\left( \frac{\partial }{\partial t}  \mathbf{X}(\cdot, t) \right)^\bot = \mathbf{H}(\cdot,t),
\end{equation}
for all ancient time $t \in \left(-\infty, t_{0} \right)$. Here $\mathbf{H}(\cdot,t)$ denotes the mean curvature vector for the time slice ${\Sigma}_{t}$. It follows that self-shrinkers correspond to ancient solutions to MCF that evolve over time by homotheties.

\subsection{Monotonicity of the Gaussian area}
 
Consider the backward heat kernel $\rho\left( \mathbf{X}, \, t\right)=\frac{1}{ { \left(4 \pi \left( t_{0} - t\right) \,\right)}^{\frac{n}{2}} } e^{ - \, \frac{ {\vert \mathbf{X} \vert}^{2} }{ 4 \left(t_{0} -t \right)}  }$  for the backward heat equation $f_{t}=-{\Delta}_{{\mathbb{R}}^{n+1}} f$. Huisken~\cite[Section 3]{Huisken 1990} showed that closed hypersurfaces evolving under MCF, for $t < t_{0}$, satisfy
\begin{equation} \label{Huisken's monotonicity}
 \frac{\partial}{\partial t} \int_{\Sigma_t}   \rho   = - \int_{\Sigma_t}  \rho \,    
 { \left\vert \, \mathbf{H} + \frac{1}{2\left(t_{0}-t \, \right)}  {\mathbf{X}}^{\bot}    \right\vert}^{2}  \leq 0.
\end{equation}
It follows from this monotonicity formula that mean curvature flow behaves asymptotically like a self-shrinker at a singularity where the curvature does not blow-up too fast~\cite[Theorem 3.5]{Huisken 1990}. Also, see Hamilton's monotonicity formula \cite{Hamilton 1993} and the local monotonicity formula due to Ecker \cite{Ecker 2004}.

\subsection{Minimality of self-shrinkers}   \label{variational eqns}
The first variation formula for weighted area (Ilmanen's lecture notes \cite[Section 2]{Ilmanen 1998} and Morgan's book \cite[Chapter 18]{Morgan 2009}) shows that self-shrinkers are variational objects. A submanifold $\Sigma$ immersed in a Riemannian manifold $\left(\mathcal{M}, g\right)$ with density  $e^{\Psi}$ is minimal if and only if its weighted mean curvature vector  ${\mathbf{H}}_{\Psi}={\mathbf{H}} - {\left( {\nabla}_{\mathcal{M}} \Psi  \right)}^{\bot}$
vanishes, where $\mathbf{H}$ denotes the mean curvature vector field on $\Sigma$ in $\left(\mathcal{M}, g \right)$ and  $ {\left( {\nabla}_{\mathcal{M}} \Psi  \right)}^{\bot}$ is the orthogonal projection of the vector field ${\nabla}_{\mathcal{M}} \Psi$ into the normal bundle of $\Sigma$. Given a hypersurface $\Sigma$ in ${\mathbb{R}}^{n+1}$ and constant $\alpha<0$, the following three statements are equivalent to each other:

\begin{itemize}
\item[1.] $\Sigma$ is critical for the Gaussian area functional ${\int}_{\Sigma}  \; e^{ \frac{ \alpha \vert   \mathbf{X}  \vert^{2} }{2} } {dvol}_{\Sigma}$.

\item[2.] $\Sigma$ satisfies the Euler-Lagrange equation $ {\Delta}_{  {\mathbf{g}}_{{}_{\Sigma}}  }  \mathbf{X} = {\left[ \, {\nabla}_{{\mathbb{R}}^{n+1}} \left( \frac{  \alpha \, {\vert \mathbf{X} \vert}^{2} }{2} \; \right) \, \right]}^{\bot} = \alpha \mathbf{X}^{\bot}$ for the Gaussian area. 

\item[3.] $\Sigma$ is a minimal submanifold in the Riemmanian manifold ${\mathbb{R}}^{n+1}$ with metric $e^{ \frac{ \alpha \vert   \mathbf{X}  \vert^{2} }{n} } \left( {dx_{1}}^{2}+\cdots+{dx_{n+1}}^{2} \right)$.
\end{itemize}

\subsection{Minimal cones as self-shrinkers}

The Clifford cone ${x_{1}}^{2}+{x_{2}}^{2}={x_{3}}^{2}+{x_{4}}^{2}$ over the Clifford torus in  ${\mathbb{S}}^{3}$ is a self-shrinker in ${\mathbb{R}}^{4}$. More generally, whenever ${\mathcal{S}}^{n-1}$ is a minimal hypersurface in the round hypersphere ${\mathbb{S}}^{n}$, its cone $C\left( \mathcal{S} \right):=\left\{ r \mathbf{p} \, : \, r \in \mathbb{R}, \mathbf{p} \in \mathcal{S} \right\}$ becomes a minimal hypersurface in ${\mathbb{R}}^{n+1}$. Since ${\mathbf{X}} \cdot  \mathbf{N} = 0$ on $C\left( \mathcal{S} \right)$, it follows that a minimal cone is a self-shrinker satisfying $H = 0 =\alpha \, {\mathbf{X}} \cdot  \mathbf{N}$ for any coefficient $\alpha$.

\subsection{Normalization of the coefficient}

Without loss of generality, we can use dilations to normalize the coefficient $\alpha<0$. In the literature, two normalizations $\alpha=-1$ (for instance, \cite{Huisken 1990}) and $\alpha=-\frac{1}{2}$ (for instance, \cite{Colding Minicozzi Pedersen 2015}) are common.  Unless otherwise noted, we take the normalization $\alpha = -\frac{1}{2}$ for the remainder of the survey.


\section{Results on existence and uniqueness of closed self-shrinkers}   \label{Section 3} 

\subsection{Shrinking curves in the plane}

In 1956, Mullins \cite{Mullins 1956} introduced the one-dimensional mean curvature flow, \emph{the curve shortening flow}, in $\mathbb{R}^2$ and constructed examples of solitons for the flow. In 1986, Gage and Hamilton \cite{Gage Hamilton 1986} solved the shrinking conjecture by showing 
that a convex curve collapses to a \emph{round point} under the curve shortening flow. The curve remains convex and becomes  \emph{circular} as it shrinks, in the sense that the ratio of the inscribed radius to the circumscribed radius approaches $1$, the ratio of the maximum curvature to the minimum curvature approaches $1$, and the higher order derivatives of the curvature converge to $0$ uniformly.

In 1987, Grayson \cite{Grayson 1987} proved the striking result that an embedded, not necessarily convex, closed curve eventually becomes convex, and therefore it eventually contracts to a round point, under the curve shortening flow.  In 1998, Huisken gave a concise proof \cite{Huisken 1998} of Grayson's Theorem, using a distance comparison argument and the classification and characterization of solitons as asymptotic models for singularities. We refer the interested reader to recent proofs by Andrews-Bryan \cite{Andrews Bryan 2011a, Andrews Bryan 2011b} and Magni-Mantegazza \cite{Magni Mantegazza 2014}.

In the mid 1980's, Abresch-Langer~\cite{Abresch Langer 1986} and Epstein-Weinstein~\cite{Epstein Weinstein 1987} independently investigated the self-shrinking solitons for the curve shortening flow.  Here are numerical approximations of some of self-shrinkers:
 \begin{figure}[H]
 \centering
 \includegraphics[height=4cm]{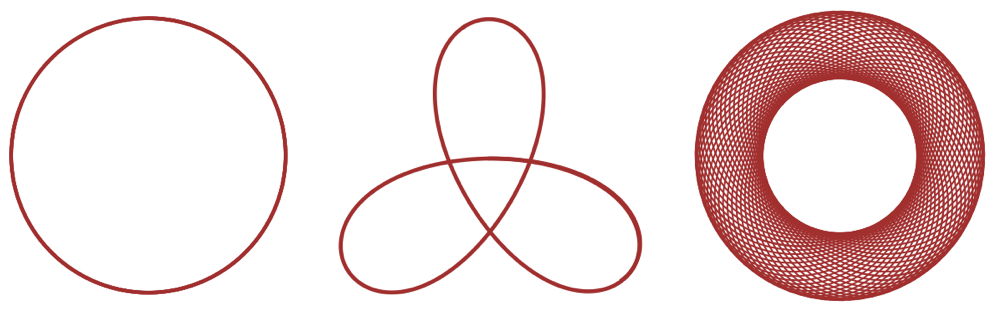}
 \caption{Examples of self-shrinkers for the curve shortening flow.}
 \end{figure}
 Unlike higher dimensional cases, the one-dimensional self-shrinker equation admits explicit first integrals. 
 
 \begin{theorem}[\textbf{Geometric conservation laws for self-shrinking curves}, \cite{Abresch Langer 1986, Epstein Weinstein 1987, Halldorsson 2012}]
  Let the function $\kappa$ denote the curvature of an immersed non-flat self-shrinker with the coefficient $\alpha<0$ in the $xy$-plane.
 \begin{enumerate}
 \item[1.] The quantity $\kappa \, e^{ \alpha  \left(   \frac{{x}^{2}+{y}^{2}}{2}   \right) }$ is constant. Hence, the curvature $\kappa$ is an 
 increasing function of the radius $\sqrt{x^2 +y^2}$.
 \item[2.] The entropy ${\kappa}^{2} +  \alpha \ln \left( {\kappa}^{2} \right)+ {\left(\frac{d\kappa}{d\theta} \right)}^{2}$ is constant, 
 where $\theta$ is the angle between the tangent vector and the $x$-axis.
 \end{enumerate}
 \end{theorem}
  
 \begin{proof} 
 Let ${\mathbf{X}}(s)=\left( x(s), y(s) \right)$ denote an immersed non-flat 
 self-shrinker parameterized by an arc length $s$. We introduce the angle function $\theta(s)$ between the unit tangent vector $T(s)=\frac{d{\mathbf{X}}}{ds}$ and the $x$-axis, 
 the unit normal vector $\mathbf{N}(s) = \mathrm{J} \, \mathbf{T}(s)$ (with the $\frac{\pi}{2}$-rotation $ \mathrm{J}$), the signed curvature $\kappa(s)=\frac{d\theta}{ds}$,
 tangential support function $\tau(s)={\mathbf{X}}(s) \cdot  \mathbf{T}(s)$, and normal support function $\nu(s)={\mathbf{X}}(s) \cdot  \mathbf{N}(s)$. 
Combining the self-shrinker equation $ \kappa (s) = \alpha \, \nu(s)$ with the coefficient $\alpha<0$ 
 and the structure equations for the curve 
\begin{equation} \label{structure 01}
\frac{d \tau}{ds} = 1+ \kappa \nu, \quad \frac{d \nu}{ds} = -\kappa \tau, \quad \frac{d}{ds} \left( \frac{{\tau}^{2}+{\nu}^{2}}{2} \right) = \tau, \quad x^2 +y^2 = {\tau}^{2}+{\nu}^{2}
\end{equation}
implies the conservation law
\[
    \frac{d}{ds} \left[    \kappa \, e^{ \alpha  \left(   \frac{{x}^{2}+{y}^{2}}{2}   \right) } \right]
 = \frac{1}{\alpha} \, \frac{d}{ds} \left( \nu \, e^{ \alpha  \left(   \frac{{\tau}^{2}+{\nu}^{2}}{2}    \right) } \right) =  \frac{1}{\alpha}   \left( -\kappa + \alpha \nu \right) \, \tau \, e^{  \alpha  \left(    \frac{{\tau}^{2}+{\nu}^{2}}{2}   \right)  } =0,
\]
which guarantees that the geometric quantity $\kappa \, e^{ \alpha  \left(   \frac{{x}^{2}+{y}^{2}}{2}   \right) }$ is constant on the self-shrinker (\cite[Lemma 5.5]{Halldorsson 2012} for $\alpha=-1$ and \cite[Theorem A]{Abresch Langer 1986}).
Reading the first two structure equations (\ref{structure 01}) with respect to the angle $\theta$ yields 
\begin{equation} \label{structure 01}
\frac{d \tau}{d\theta} = \frac{1}{\kappa} + \nu, \quad \frac{d \nu}{d \theta} = - \tau, \quad \frac{1}{\kappa} = - {\nu}_{\theta \theta} - \nu.
\end{equation}
Combining these and the self-shrinker equation $ \kappa = \alpha \, \nu$ implies the conservation law (\cite[Section 1]{Epstein Weinstein 1987} for  $\alpha=-1$):
\[
    \frac{d}{d \theta} \left[    {{\kappa}_{\theta}}^{2} + {\kappa}^{2} + \alpha \ln \left( {\kappa}^{2} \right)   \right]
    = 2  {\kappa}_{\theta} \left[ {\kappa}_{\theta \theta} + \kappa + \frac{\alpha}{\kappa}  \right] =  2{\kappa}_{\theta}\,  \alpha   \left[   \left( {\nu}_{\theta \theta} + \nu \right) + \frac{1}{\kappa}  \right] =0.
\]
\end{proof} 

Recently, classification results and conservation laws for self-shrinkers were developed in other geometric contexts by Halldorsson \cite{Halldorsson 2012, Halldorsson 2015} and Chang \cite{Chang 2017}.

\subsection{Rigidity results for self-shrinkers}

Round spheres admit geometric characterizations both as constant mean curvature (CMC) surfaces and as self-shrinkers for the mean curvature flow. The classical theorems of Jellett, Alexandrov, and Hopf show that round spheres posses some rigidity as CMC hypersurfaces.  Jellett's Theorem in  ${\mathbb{R}}^{3}$ and its generalization \cite{Gimeno 2015, Hsiung 1956, Hsiung 1959, Kwong 2016, Montiel 1999}, which uses Hsiung-Minkowski integral formulas \cite{Hsiung 1956, Hsiung 1959}, confirms that a closed, star-shaped, CMC hypersurface is round. 
 Alexandrov used his method of moving planes to prove that an embedded, closed CMC hypersurface in ${\mathbb{R}}^{n+1}$ must be a round sphere. 
 The embedded assumption is essential due to the existence of immersed tori in ${\mathbb{R}}^{3}$ with positive constant mean curvature, see Abresch \cite{Abresch 1987} 
 and Wente \cite{Wente 1986}. Hopf showed if a closed immersed CMC surface in ${\mathbb{R}}^{3}$ is a topological sphere, then it must be round. One of Hopf's two proofs \cite{Hopf 2003} exploits a beautiful fact that the Codazzi equation implies the existence of globally well-defined holomorphic quadratic differential on CMC surfaces. The proof can be generalized to a wider class of surfaces in more general ambient spaces, for instance, as in \cite{Abresch Rosenberg 2004, Bryant 2011, Fraser Schoen 2015}.

Despite the similarity between the CMC and self-shrinker equations, the classical rigidity results of Alexandrov and Hopf for CMC hypersurfaces do not hold for self-shrinkers. (There are examples of an embedded $\mathbb{S}^1 \times \mathbb{S}^{n-1}$ self-shrinker and an immersed, non-round $\mathbb{S}^n$ self-shrinker in $\mathbb{R}^{n+1 \geq 3}$.) In addition, analogues of the classical Weierstrass-Enneper representation (holomorphic resolution of minimal surfaces) or Kenmotsu representation (which prescribes harmonic Gauss map  of CMC surfaces \cite {Kenmotsu 1979}) are not known for self-shrinkers. However, just as in the CMC setting, round spheres do posses some rigidty as closed self-shrinkers.

In 1984, Huisken \cite{Huisken 1984} established that a convex hypersurface in ${\mathbb{R}}^{n+1 \ge 3}$ shrinks to a round point by showing that the hypersurface, under a rescaled flow, converges to a totally umbilical hypersurface. Hence, Huisken's result in ${\mathbb{R}}^{n+1 \ge 3}$ is a higher dimensional analogue of the Gage-Hamilton Theorem for the curve shortening flow in ${\mathbb{R}}^{2}$. Since self-shrinkers keep their shape under the flow, these \emph{parabolic asymptotic convergence results} implicitly imply the \emph{elliptic rigidity result} that a closed, convex self-shrinker in any dimension is round. We highlight two additional rigidity results for closed self-shrinkers:

\begin{itemize}
\item[]
\item \emph{Rigidity of spheres as mean-convex self-shrinkers in ${\mathbb{R}}^{n+1 \ge 3}$}: In 1990, Huisken  \cite[Theorem 4.1]{Huisken 1990} showed that a closed, mean-convex ($H>0$) self-shrinker must be a round sphere. The key starting point in his argument is to combine the self-shrinker equation and Simons' identity for the squared length ${\vert A \vert}^{2}$ of the second fundamental form to obtain an explicit expression for the Laplacian of the well-defined quotient function $\frac{ {\vert A \vert}^{2} }{ H^2 }$. 
Since the self-shrinker is compact, the maximum principle guarantees that $\frac{ {\vert A \vert}^{2} }{ H^2 }$ is constant. Subsequent analysis of the pde 
for $\frac{ {\vert A \vert}^{2} }{ H^2 }$ and Hsiung-Minkowski integral formulas ultimately lead to the conclusion that the mean curvature $H$ is a positive constant. Now, a mean-convex 
self-shrinker has positive support function. Therefore, the self-shrinker is round. See also Montiel's Theorem \cite{Montiel 1999}.

\item[]
\item \emph{Rigidity of spheres as embedded ${\mathbb{S}}^{2}$ self-shrinkers in ${\mathbb{R}}^3$}:  In 2016, Brendle \cite{Brendle 2016} proved the long-standing Alexandrov-Hopf type conjecture showing that an embedded, topological $\mathbb{S}^2$ self-shrinker in ${\mathbb{R}}^3$ must be a round sphere. Unlike in the CMC case, combining the self-shrinker equation with the Codazzi equations does not produce a holomorphic quadratic differential on self-shrinkers, so the Hopf type approach does not directly work for self-shrinkers. The key result in Brendle's proof is that the sign of the normal support function does not change on an embedded ${\mathbb{S}}^{2}$ self-shrinker in ${\mathbb{R}}^3$, i.e., the closed, embedded self-shrinker is star-shaped. Thus, the self-shrinker is mean-convex (by the definition of the self-shrinker equation), and by Huisken's Rigidity Theorem it is a round sphere.  
\end{itemize}

\bigskip

\subsection{Examples of self-shrinkers}
\label{examples}

Though round spheres are rigid as CMC hypersurfaces and self-shrinkers under certain additional assumptions, there are numerous examples that contrast the rigidity results from the previous section. Hsiang-Teng-Yu \cite{Hsiang Teng Yu 1983} proved that there exist infinitely many distinct CMC immersions of ${\mathbb{S}}^{2k-1}$ in ${\mathbb{R}}^{2k \geq 4}$, see also \cite{Hsiang 1982}. These examples come from studying closed hypersurfaces with bi-rotational symmetry. (We consider the problem of constructing closed self-shrinkers with bi-rotational symmetry in Section~\ref{Section 6}.) A rich source of examples of immersed self-shrinkers comes from hypersurfaces with rotational symmetry. Using the shooting method for geodesics (see Section~\ref{Section 5}), an infinite number of complete, self-shrinkers for each of the rotational topological types: $\mathbb{S}^{n}$, $\mathbb{S}^1 \times \mathbb{S}^{n-1}$, $\mathbb{R}^{n}$, and $\mathbb{S}^1 \times \mathbb{S}^{n-1}$ were constructed in~\cite{Drugan Kleene 2017}. 

In the following, we introduce the geodesic equation for the profile curve of a self-shrinker with rotational symmetry and highlight a few modern examples of closed self-shrinkers. We note that even though rotational self-shrinkers have a variational characterization, it is unknown if the geodesic equation is integrable.

\begin{itemize}
\item[]
\item \emph{Profile curves of rotational shrinkers as geodesics in the half-plane}: Self-shrinkers are minimal submanifolds in the Remannian manifold $\mathbb{R}^{n+1 \geq 3}$ equipped with the conformal metric $e^{ - \frac{\vert   \mathbf{X}  \vert^{2} }{2n} } \left( {dx_{1}}^{2}+\cdots+{dx_{n+1}}^{2} \right).$ Applying the first variation formula to a rotational hypersurface $\mathbf{X}(s,\omega) =  \left( x(s), r(s) \omega \right)$, $s \in \mathbb{R}$, $
\omega \in \mathbb{S}^{n-1}$, with the profile curve $(x(s), r(s))$ in the half-plane $\mathbb{H} = \{ (x,r)\in\mathbb{R}^2 \, | \, r>0 \}$, shows that it is a self-shrinker if and only if its profile curve is a geodesic for the conformal metric $$g_{Ang} = r^{2(n-1)} e^{- \frac{x^2 + r^2}{2}} (dx^2 + dr^2).$$ The geodesic equation for the profile curve $(x(s),r(s))$ of a self-shrinker with rotational symmetry is
\[
\frac{x' r'' - x'' r'}{x'^2 + r'^2}  = \left( \frac{n-1}{r}  -\frac{r}{2} \right) x' + \frac{1}{2}xr',
\]
and the Gauss curvature of the Riemannian manifold $(\mathbb{H}, g_{Ang})$ is given by $$K =  \frac{r^2 + (n-1)}{r^{2n}} e^{ \frac{x^2+r^2}{2}} > 0.$$

\item[] 
\item \emph{The fundamental examples of rotational self-shrinkers}: The sphere of radius $\sqrt{2n}$ centered at the origin, a flat plane through origin, and a round cylinder of radius $\sqrt{2(n-1)}$ with axis through the origin are examples of self-shrinkers with rotational symmetry. We have the following profile curves for these fundamental examples:

\begin{itemize}

\item[1.] \emph{Round sphere}: $x^2 + r^2 = 2n$ with the profile curve $(x(s),r(s)) = \left(\sqrt{2n} \cos \left(\frac{s}{\sqrt{2n}}\right) , \sqrt{2n} \sin\left(\frac{s}{\sqrt{2n}}\right)\right)$.

\item[2.] \emph{Flat plane}: $x \equiv 0$ with the profile curve $(x(s),r(s)) = (0,s)$.

\item[3.] \emph{Round cylinder}: $r\equiv \sqrt{2(n-1)}$ with the profile curve $(x(s),r(s)) = (s, \sqrt{2(n-1)})$.
\end{itemize}

\item[] 
\item \emph{Angenent's torus}: Using the \emph{shooting method} for geodesics, Angenent~\cite{Angenent 1992} gave the first proof of the existence of an embedded torus ($\mathbb{S}^1 \times \mathbb{S}^{n-1}$) self-shrinker in $\mathbb{R}^{n+1}$ with a rotational symmetry. See also  \cite{Drugan 2015, Drugan Nguyen preprint,  Moller 2011}.

 \begin{figure}[H]
 \centering
 \includegraphics[height=4cm]{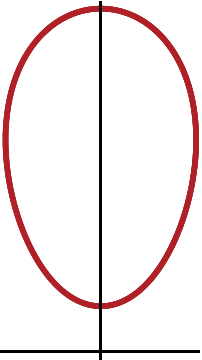}
 \caption{The profile curve whose rotation about the horizontal axis is an embedded torus self-shrinker.}
 \end{figure}

\item[] 
\item \emph{Immersed sphere self-shrinker}: Motivated by Angenent's construction and using the shooting method from the axis of rotation, it was shown in \cite{Drugan 2015} that there exists an immersed and non-embedded $\mathbb{S}^n$ self-shrinker in $\mathbb{R}^{n+1}$ with a rotational symmetry (see Section~\ref{immersed sphere} for a detailed sketch of the proof). The existence of this immersed self-shrinkers explains why the embeddedness assumption is essential in Brendle's rigidity result for embedded $\mathbb{S}^2$ self-shrinkers in $\mathbb{R}^{3}$ \cite{Brendle 2016}.
 \begin{figure}[H]
 \centering
 \includegraphics[height=4cm]{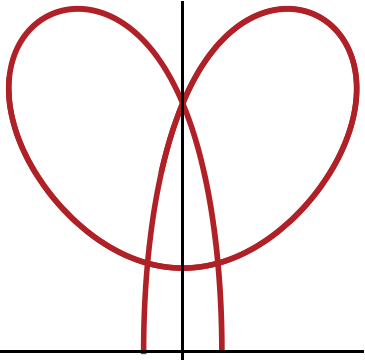}
 \caption{The profile curve whose rotation about the horizontal axis is an immersed sphere self-shrinker.}
 \label{fig:is}
 \end{figure}

\item[] 
\item \emph{Immersed self-shrinkers}: Building on the work in \cite{Angenent 1992, Drugan 2015, Kleene Moller 2014}, infinitely many immersed and non-embedded self-shrinkers for each of the rotational topological types: $\mathbb{S}^n$, $\mathbb{S}^1 \times \mathbb{S}^{n-1}$, $\mathbb{R}^{n}$, and $\mathbb{S}^1 \times \mathbb{R}^{n-1}$ were constructed in~\cite{Drugan Kleene 2017}. The main idea for the construction is to study the behavior of solutions to the geodesic equation near two known self-shrinkers and use continuity arguments to find complete self-shrinkers between them. See Section~\ref{immersed shrinkers} for illustrations on how to carry out this heuristic.

 \begin{figure}[H]
 \centering
 \includegraphics[height=4cm]{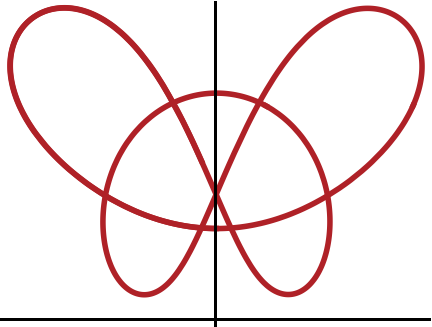}
 \caption{The profile curve whose rotation about the horizontal axis is an immersed torus self-shrinker.}
 \end{figure}
  
\item \emph{M{\o}ller's embedded shrinkers with higher genus in $\mathbb{R}^{3}$}: In 2011, M{\o}ller \cite{Moller 2011} performed a smooth desingularization of two rotational self-shrinkers: Angenent's torus and the round sphere. M{\o}ller's shrinkers are generalizations of Costa's embedded three-end minimal surface \cite{Hoffman 1987}, which can be viewed as a smooth desingularization of two rotational minimal surfaces: a catenoid and a plane passing the neck of the catenoid. More concretely, M{\o}ller proved the existence of a large lower bound $N_{0}$ such that for each even $g = 2k \geq N_0$ there exists a closed self-shrinker ${\Sigma}_{g}$ in $\mathbb{R}^{3}$ with  genus $g$ that is invariant under the dihedral symmetry group with $2g$ elements. Furthermore, the sequence of self-shrinkers ${\Sigma}_{g}$ converges in the Hausdorff sense (and smoothly away from the two initial intersection circles) to the union of Angenent's torus and the round sphere.

 \begin{figure}[H]
 \centering
 \includegraphics[height=4.785cm]{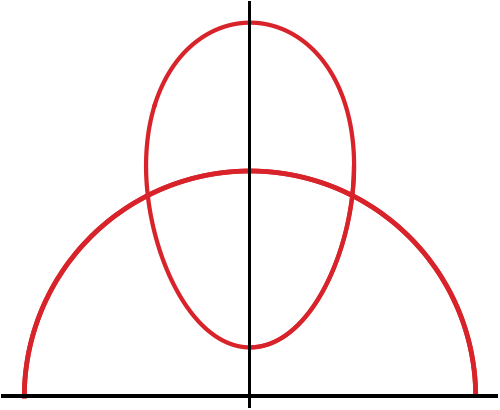}
   \caption{Two intersecting geodesics in M{\o}ller's desingularization of self-shrinking tori and round sphere}
 \end{figure}

\end{itemize}


\section{Variational method for an embedded $\mathbb{S}^1 \times \mathbb{S}^{n-1}$ self-shrinker in $\mathbb{R}^{n+1}$}    \label{Section 4} 
 
In this section, we outline the parabolic proof from \cite{Drugan Nguyen preprint} of the existence of a rotational self-shrinking torus. The proof uses variational techniques, applied to the geodesic problem from Section~\ref{examples}, to find a simple, closed geodesic and gives an estimate for the length of this geodesic in $(\mathbb{H},g_{Ang})$. It is unknown if this proof recovers the closed geodesics constructed in~\cite{Angenent 1992}. (See Section \ref{Section 7}.)  

\begin{theorem}[\cite{Drugan Nguyen preprint}]
\label{thm:variational}
For $n \geq 2$, there exists a simple, closed geodesic $\gamma_\infty$, for the conformal metric $$g_{Ang} = r^{2(n-1)} e^{-(x^2 + r^2)/2} (dx^2 + dr^2)$$ on the half-plane $\mathbb{H} = \{ (x,r) \in \mathbb{R}^2 \, | \, r > 0 \}$.  Moreover, its length $L_n(\gamma_{\infty})$ in the metric $g_{Ang}$ is less than the length of the double cover of the half-line $x=0$: 
	\[
	L_n(\gamma_{\infty}) := \int_{\mathbb{S}^1} r^{n-1} e^{-(x^2 + r^2)/4} \sqrt{x'^2 + r'^2}du 
	< 2 \int_0^\infty s^{n-1} e^{-s^2/4} ds.
	\]
\end{theorem}

The idea for finding this closed geodesic is to study a modified curve shortening flow:
\begin{equation}
\label{eq:mcsf}
\frac{\partial}{\partial t} \gamma_t = \frac{k_g}{K}\bf{n},
\end{equation}
where $k_g$ is the geodesic curvature and $K>0$ is the Gauss curvature in $(\mathbb{H},g_{Ang})$. The goal is to create an initial curve $\gamma_0$ whose evolution under the flow converges to the geodesic $\gamma_\infty$. To do this we consider a special family of initial curves and study their evolutions under the modified curve shortening flow. The advantage of this approach is that the evolution is well-known and the crux of the proof is in selecting the appropriate family of initial curves, all of which have length less than twice the length of the half-line $x=0$ and enclose Gauss area of exactly $2 \pi$.

The modified curve shortening flow has two important properties.
\begin{itemize}

\item[1.] \emph{The flow decreases length}: Since the arc length $ds$ evolves according to $\frac{\partial}{\partial t} ds = - \frac{k_g^2}{K} \, ds$,
 the length $L_n(\gamma_t)$ is non-increasing:
\[
\frac{d}{dt} L_n(\gamma_t) = -\int_{\gamma_t} \frac{k_g^2}{K}\, ds \leq 0.
\]

\item[2.] \emph{The flow preserves total Gauss area of $2\pi$}: When the evolving curves $\gamma_t$ are simple, closed curves bounding domains $\Omega_t$, the Gauss-Bonnet formula gives 
\begin{align*}
\frac{d}{dt} \iint_{\Omega_t} KdA 
= -\oint_{\gamma_t} k_g\,  ds 
= -2\pi + \iint_{\Omega_t} KdA.
\end{align*}
In particular, if the total Gauss area enclosed by the initial curve equals $2 \pi$, then the total Gauss area enclosed by $\gamma_t$ is also $2 \pi$ as long as the flow exists.

\end{itemize}

Working in regions where the Gaussian curvature is uniformly bounded from above and below away from 0, the short-time existence for the flow follows from~\cite{Gage 1990, Angenent 1990}. This gives the following long-time existence result.
	
\begin{proposition}[Long-time existence]
\label{prop:long-time-existence}
Let $\gamma_0$ be a simple closed curve. If the domain $\Omega_0$ enclosed by $\gamma_0$ satisfies $\iint_{\Omega_0} K dA = 2 \pi$, then the evolution of $\gamma_0$ with normal velocity $k_g/K$ exists for all time.  
\end{proposition}

By selecting a family of initial curves, all of which have length less than twice the length of the half-line $x=0$ and enclose total Gauss area of exactly $2 \pi$, Proposition \ref{prop:long-time-existence} guarantees that the evolutions of these curves exist for all time, and because the length decreases, the flow starting from each curve does not converge to a double cover of the half-line $\{x=0\}$. To select the family of initial curves, we consider rectangles  $R[a,b,c]$ with vertices $(a, -c), (a, c), (b,c), (b,-c)$, $a<b$ and $c>0$. 

\begin{figure}[htb]
\begin{picture}(120,90)(0,0)
\linethickness{0.2mm}
\put(10,10){\line(0,1){80}}
\put(10,90){\vector(0,1){0.12}}
\linethickness{0.25mm}
\put(10,50){\line(1,0){110}}
\put(120,50){\vector(1,0){0.12}}
\linethickness{0.25mm}
\multiput(40,30)(0,1.95){21}{\line(0,1){0.98}}
\linethickness{0.2mm}
\multiput(40,30)(1.96,0){26}{\line(1,0){0.98}}
\linethickness{0.2mm}
\multiput(90,30)(0,1.95){21}{\line(0,1){0.98}}
\linethickness{0.2mm}
\multiput(40,70)(1.96,0){26}{\line(1,0){0.98}}
\put(15,85){\makebox(0,0)[cc]{$x$}}

\put(115,45){\makebox(0,0)[cc]{}}

\put(115,45){\makebox(0,0)[cc]{}}

\put(45,75){\makebox(0,0)[cc]{}}

\put(115,45){\makebox(0,0)[cc]{$r$}}

\put(35,45){\makebox(0,0)[cc]{$a$}}

\put(95,45){\makebox(0,0)[cc]{$b$}}

\put(3,75){\makebox(0,0)[cc]{$c$}}

\put(1,35){\makebox(0,0)[cc]{$-c$}}

\put(10,70){\makebox(0,0)[cc]{}}

\linethickness{0.2mm}
\linethickness{0.2mm}
\linethickness{0.2mm}
\put(8,70){\line(1,0){4}}
\linethickness{0.2mm}
\put(8,30){\line(1,0){4}}
\end{picture}
\caption{The rectangle $R[a,b,c]$.} 
\end{figure}
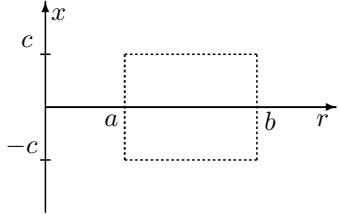

Numerics show that all the rectangles with $c=1/2$ will satisfy the requirement on length. To simplify computations, we take $c=c_0 \approx 0.481$ to be the positive real number such that $\frac{e^{-c_0^2/4}}{\int_0^{c_0} e^{-x^2/4}dx}=2$.

\begin{proposition}
\label{prop:PerimeterRectangle}
For every $a,b \in (0, \infty)$ and $n\geq 2$, we have
	\begin{equation}
	\label{eq:PerimeterBounded}
	L_n(a,b,c_0) <  2 \int_0^\infty s^{n-1} e^{-s^2/4} ds.
	\end{equation}
\end{proposition}

The proof of this result is by induction on the dimension and numerical verification on lower dimensions. The induction is quite involved and makes heavy use of Taylor expansions. 

\begin{proposition}
There is a smooth function $\varphi: \mathbb{R}_+ \to \mathbb{R}_+$ with $\varphi(a)>a$, with $\lim_{a\to 0} \varphi(a) = 0$, so that  the family of rectangles $R[a,\varphi(a),c_0]$ satisfies 
	\[
	  \iint_{R[a,\varphi(a),c_0]} K dA = 2 \pi, \quad L_n(a,\varphi(a),c_0) <  2 \int_0^\infty s^{n-1} e^{-s^2/4} ds.
	\]
\end{proposition}

Next we consider the modified curve shortening flow for the initial curves $ R[a,\varphi(a),c_0]$.

\begin{proposition}  Let $\Phi: \mathbb{R}_+ \times \mathbb{R}_+ \to C^{0} (\mathbb S^1, \mathbb{H})$ be the map with the following properties:
	\begin{enumerate}
	\item $\Phi(a, 0) = R[a,\varphi(a), c_0]$,
	\item For fixed $a$, $\Phi(a, t)$ satisfies the evolution equation \eqref{eq:mcsf}.
	\end{enumerate}
There exists an $a_0 \in \mathbb{R}_+$ so that $\Phi(a_0,t)$ intersects the line $r = \sqrt{2(n-1)}$ for all time $t \in \mathbb{R}_+$. 
\end{proposition}	

\begin{proof} The set of curves that do not intersect the line $r = \sqrt{2(n-1)}$ is split into two disjoint sets
 	\begin{align*}
	A_1 & = \{ \text{continuous closed curves } \gamma: \mathbb{S}^1 \to \mathbb{H} \mid \gamma(s) < r_n,  s \in S^1\} ,\\
	A_2 & = \{ \text{continuous closed curves } \gamma: \mathbb{S}^1 \to \mathbb{H} \mid \gamma(s) > r_n,  s \in S^1\}. 
	\end{align*}
The line $r=\sqrt{2(n-1)}$ is a geodesic (it is the profile curve for the round cylinder self-shrinker), so it is stationary under the modified curve shortening flow. By the maximum principle, if $\Phi(a,t_0) \in A_i$ for some $i=1,2$ and some time $t_0$, then $\Phi(a, t) \in A_i$ for $t \geq t_0$. Consider the following subsets of $\mathbb{R}_+$:
	\begin{align*}
	U_i & = \{ a \in \mathbb{R}_+ \mid \exists t >0, \Phi(a,t) \in A_i\}
	\end{align*}
Both $U_1$ and $U_2$ are open and $U_1 \cap U_2 = \emptyset$. Therefore $U_1 \cup U_2 \neq \mathbb{R}_+$, which proves the existence of $a_0$.
\end{proof} 

The remainder of the proof of Theorem~\ref{thm:variational} is dedicated to showing that the curves $\gamma_t:= \Phi(a_0,t)$ converge to a simple, closed geodesic that is symmetric about the $r$-axis and convex in the Euclidean sense. First, we show that on compact sets, the curves $\gamma_t$ approach a geodesic along a subsequence. 

\begin{proposition}
\label{prop:sequenceti} There is a sequence $t_{i}$ so that 
	\begin{equation}
	\label{eq:TotalCurvatureDecay}
	\int_{\gamma_{t_i}\cap E}  |k|\, ds \to 0
	\end{equation}
for any compact subset $E$ of the open half-plane $\mathbb{H}$. 
\end{proposition}

\begin{proposition}
\label{prop:C1Convergence} Let $t_i$ be the sequence from Proposition \ref{prop:sequenceti} (or possibly one of its subsequences) and let $E$ be a compact set in $\mathbb{H}$. If for some $p_i \in \mathbb{S}^1$, $i \in \mathbb N$,  the sequence $\{\gamma_{i} (p_i)\}$ converges to a point $P \in E$, then there exists a subsequence ${i_j}$ so that the connected component of $\gamma_{i_j} \cap E$ containing $\gamma_{i_j} (p_{i_j}) $ converges in $C^1$ in $E$.
The limit curve contains $P$ and satisfies the geodesic equation in $E$. 
\end{proposition}
 
For $t>0$ we know each curve $\gamma_t$ restricted to the positive quadrant is a graph over the $r$-axis.  Since the flow preserves symmetry, this follows from the initial evolution along with a result from~\cite{Angenent 1991} on intersection points. The proof is completed by showing that a convergent subsequence of $\gamma_t$ stays within a compact domain. This is done using properties of geodesics written as graphs over the $r$-axis and length estimates for $\gamma_t$. Once it is know that there is a subsequence of $\gamma_t$ that stays in a compact subset of $\mathbb{H}$, it follows that there is a subsequence of $\gamma_t$ that converges to a geodesic $\gamma_\infty$ with the desired properties.


\section{Shooting method for closed self-shrinkers with rotational symmetry}   \label{Section 5}

Recall from Section~\ref{examples} that a hypersurface with rotational symmetry $\mathbf{X}(s,\omega) =  \left( x(s), r(s) \omega \right)$, $s \in \mathbb{R}$, $\omega \in \mathbb{S}^{n-1}$ is a self-shrinker if and only if the profile curve $(x(s), r(s))$ is a geodesic in $(\mathbb{H},g_{Ang})$ (see Section~\ref{examples}). Consequently, the construction of closed self-shrinkers with rotational symmetry can be reduced to finding closed geodesics and geodesic arcs that intersect the $x$-axis orthogonally in $(\mathbb{H},g_{Ang})$.

The geodesic equation for the profile curve $(x(s),r(s))$ of a self-shrinker with rotational symmetry is
\begin{equation}
\label{geo:eqn}
\frac{x' r'' - x'' r'}{x'^2 + r'^2}  = \left( \frac{n-1}{r}  -\frac{r}{2} \right) x' + \frac{1}{2}xr'.
\end{equation}
Reparametrizing the profile curve so that $x'(s)^2 + r'(s)^2 = 1$, shows that the tangent angle 
$\alpha(s)$ solves the system
\begin{equation} 
\label{SSEq}
\left\{
\begin{array}{lll}
x'(s) & = & \cos \alpha(s), \\
r'(s) & = & \sin \alpha(s), \\
\alpha'(s) & = &\left( \frac{n - 1}{r(s)} - \frac{r(s)}{2} \right) \cos \alpha(s) + \frac{x(s)}{2} \sin \alpha(s).
\end{array}
\right.
\end{equation}
For $(x_0, r_0) \in \mathbb{H}$ and $\alpha_0 \in \mathbb{R}$, we let $\Gamma[x_0, r_0, \alpha_0](s)$ denote the unique solution to~(\ref{SSEq}) satisfying
\[
\Gamma[x_0, r_0, \alpha_0] (0) = (x_0, r_0), \quad \Gamma[x_0, r_0, \alpha_0]' (0) = (\cos(\alpha_0), \sin(\alpha_0)).
\]
The geodesics $\Gamma[x_0, r_0, \alpha_0]$ depend smoothly on the parameters $[x_0, r_0, \alpha_0] \in \mathbb{H} \times \mathbb{R}$, and as was shown in~\cite{Drugan 2015} this dependence extends smoothly to geodesics that intersect the $x$-axis orthogonally.
 

\subsection{An immersed sphere self-shrinker}
\label{immersed sphere}

In this section, we give a detailed illustration of how the shooting method for the geodesic equation~(\ref{SSEq}) can be used to prove the following result.

\begin{theorem}[\cite{Drugan 2015}]
There exists an immersed and non-embedded $\mathbb{S}^n$ self-shrinker in $\mathbb{R}^{n+1}$.
\end{theorem}

We begin by considering the one-parameter shooting problem: $$S[t] \textrm{ is the solution to~(\ref{SSEq}) with } S[t] (0) = (t, 0) \textrm{ and } S[t]' (0) = (0, 1).$$ Notice that the solutions $S[0]$ and $S[\sqrt{2n}]$ are the profile curves for a flat $\mathbb{R}^n$ self-shrinker and the round $\mathbb{S}^n$ self-shrinker, respectively. After describing the shape of $S[t]$ for small $t>0$, we will show that there is $x_* \in (0,\sqrt{2n})$ so that $S[x_*]$ intersects the $r$-axis orthogonally (see Figure~\ref{shape:sphere4}). Since the geodesic equation is symmetric with respect to reflections about the $r$-axis, the geodesic $S[x_*]$ intersects the $x$-axis orthogonally at two points (see Figure~\ref{fig:is}) and is the profile curve for an immersed $\mathbb{S}^n$ self-shrinker.

\vspace{5pt}
{\bf Step 1}: The first step in the proof is to show that $S[t]$ has the following shape for small $t>0$:

 \begin{figure}[H]
 \label{shape:sphere3}
 \centering
 \includegraphics[height=5cm]{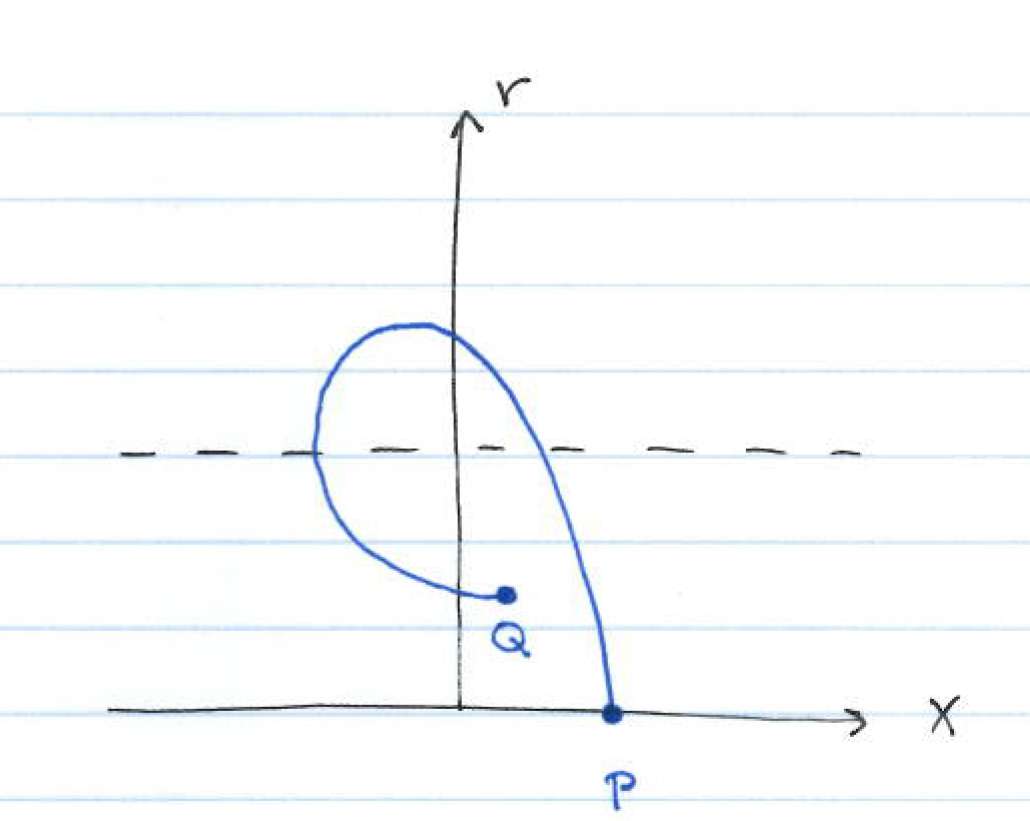}
 \caption{The shape of the geodesic $S[t]$ when $t>0$ is sufficiently small.}
 \end{figure}

In order to do this, we need to understand the behavior of the geodesic $S[t]$ as it travels away from the $x$-axis, turns around, and travels back towards the $x$-axis. Following the proof in~\cite{Drugan 2015}, we analyze $S[t]$ by writing it locally as graphs over the $r$-axis. The local graphical component $x=f(r)$ satisfies
 \begin{equation}
\label{ode:graph}
\frac{f''}{1+f'^2} = \left( \frac{r}{2} - \frac{n-1}{r} \right) f'  - \frac{1}{2}f. 
\end{equation}
Taking a derivative of~(\ref{ode:graph}), we have
\begin{equation}
\label{ode2:graph}
\frac{f'''}{1+f'^2} = \frac{2f'(f'')^2}{(1+f'^2)^2} + \left( \frac{r}{2} - \frac{n-1}{r} \right) f'' + \frac{n-1}{r^2} f'.
\end{equation}
Notice that~(\ref{ode2:graph}) is a second order differential equation for $f'$ with positive coefficient on $f'$. Much of the behavior of the geodesics $S[t]$ can be understood by analyzing these equations. We have the following results from~\cite{Drugan 2015}:

\begin{proposition}
\label{gamma:small:height}
For $t>0$, let $f_t$ denote the solution to~(\ref{ode:graph}) with $f_t(0) = t$ and $f_t'(0) = 0$. Then $f_t''<0$, and there is a point $b_t > \sqrt{2(n-1)}$ so that $\lim_{r \to b_t} f'_t(r) = -\infty$ and $f_t(b_t) > -\infty$. Moreover, there exists $\tilde{t} > 0$ so that if $t \in (0,\tilde{t} \,]$, then $$b_t \geq \sqrt{ \log{\frac{2}{\pi t^2}} },$$ $$\frac{-4(n+1)}{ \sqrt{\log{\frac{2}{\pi t^2}}} } \leq f_t(b_t) < 0,$$ and $f_t(r) < 0$, for $t \in (2\sqrt{n}, b_t]$.
\end{proposition}
This proposition tells us that when $t>0$ is sufficiently small, the first component of $S[t]$ written as a graph over the $r$-axis is concave down, decreasing, and it crosses the $r$-axis, before it blows-up at the point $B_t = ( f_t(b_t), b_t)$. (Here, the graphical component blows-up in the sense that the tangent line at $B_t$ is orthogonal to the $r$-axis.) 

 \begin{figure}[H]
 \centering
 \includegraphics[height=5cm]{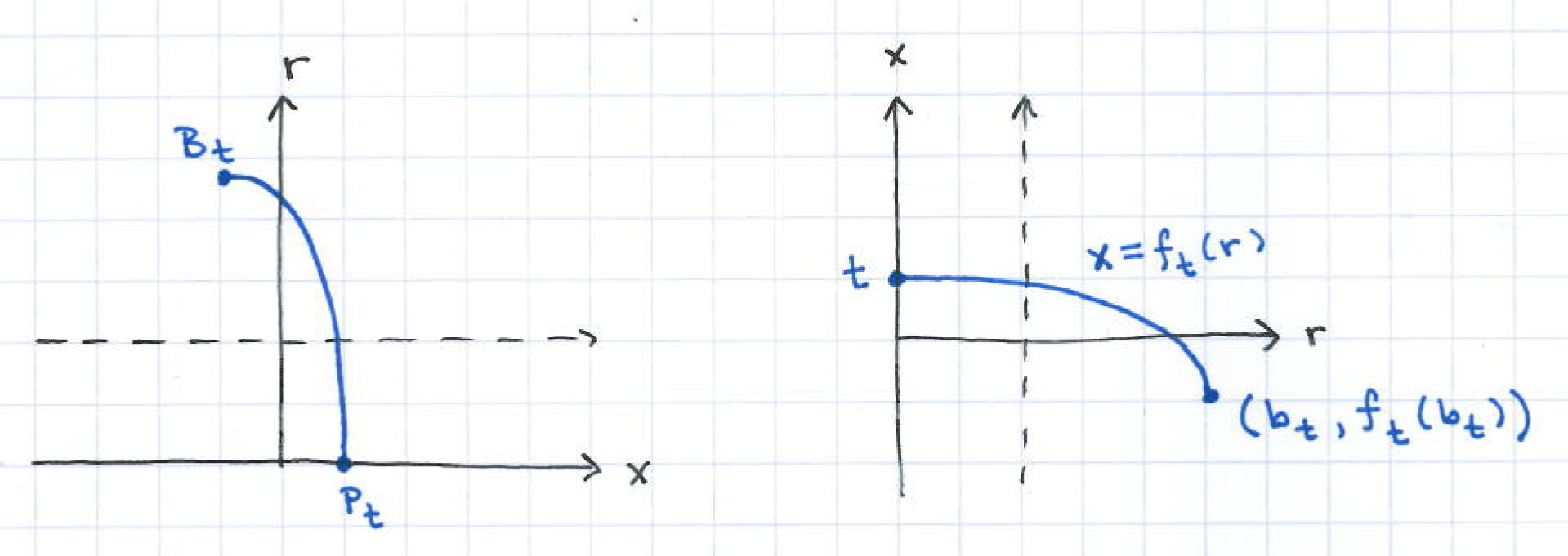}
 \caption{The initial shape of the first graphical component of $S[t]$ written as a graph over the $r$-axis when $0<t<\tilde{t}$.}
 \end{figure}

In addition, Proposition~\ref{gamma:small:height} tells us $b_t \to \infty$ and $f_t(b_t) \to 0$ as $t \to 0$. This behavior allows us to study the second component of $S[t]$ written as a graph over the $r$-axis when $t>0$ is sufficiently small.

\begin{proposition}
\label{beta:crossing}
Let $\tilde{t} >0$ be given as in the conclusion of Proposition~\ref{gamma:small:height}. For $t \in (0, \tilde{t} \,]$, let $B_t = ( f_t(b_t), b_t)$, where $f_t$ is the solution to~(\ref{ode:graph}) with $f_t(0) = t$ and $f_t'(0) = 0$, and define $g_t$ to be the unique solution to~(\ref{ode:graph}) with $g_t(b_t) = f_t(b_t)$ and $\lim_{r \to b_t} g_t'(r) = \infty$. Then there exists $0 < \bar{t} <\tilde{t}$ so that for $ t \in (0 ,\bar{t} \,]$, the solution $g_t$ has the following properties: there is $a_t \in (0,\sqrt{2(n-1)})$ so that $g_t$ is a maximally extended solution to~(\ref{ode:graph}) on the interval $(a_t, b_t)$; there is a point $c_t \in (a_t, b_t)$ so that $g_t'(c_t) = 0$; $g_t'' > 0$; and $0 < g_t(a_t) < \infty$.
\end{proposition}

This proposition tells us that when $t>0$ is sufficiently small, the second component of $S[t]$ written as a graph over the $r$-axis is concave up, it crosses the $r$-axis, and it blows-up at the point $Q_t = ( g_t(a_t), a_t)$, where $0 < a_t < \sqrt{2(n-1)}$ and $ 0 < g_t(a_t) < \infty$.

 \begin{figure}[H]
 \centering
 \includegraphics[height=5cm]{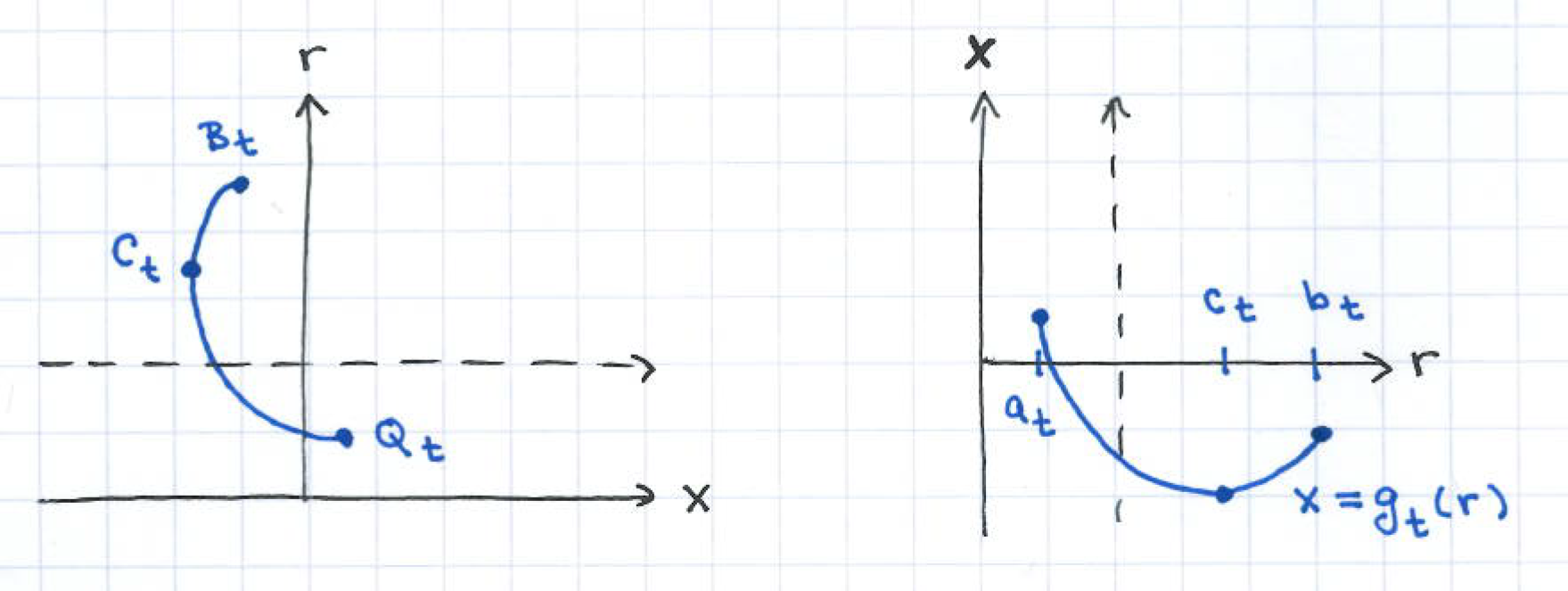}
 \caption{The initial shape of the second graphical component of $S[t]$ written as a graph over the $r$-axis when $0<t<\bar{t}$.}
 \end{figure}

Combining these two propositions shows that $S[t]$ has the initial shape illustrated in Figure~\ref{shape:sphere3} when $0<t<\tilde{t}$.

\vspace{5pt}
{\bf Step 2}: The second step is to increase the shooting parameter $t>0$ and show that there is $x_* \in (0,\sqrt{2n})$ so that $S[x_*]$ has the following shape:

 \begin{figure}[H]
 \centering
 \includegraphics[height=5cm]{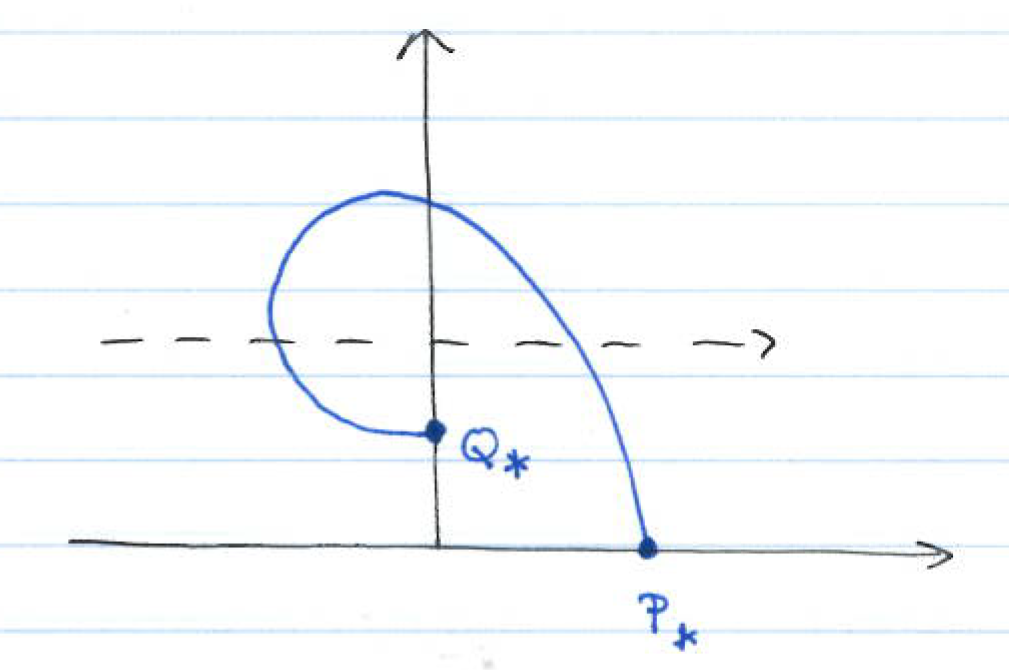}
 \caption{A sketch of the geodesic $S[x_*]$.}
 \label{shape:sphere4}
 \end{figure}

Using the notation from Step 1, we know that for $0 < t  \leq \bar{t}$, the geodesic $S[t]$ is strictly convex as it travels from $P_t = (t,0)$ to $B_t$ to $C_t$ to $Q_t$, and $Q_t$ lies in the same quadrant as $P_t$. We note that the the geodesic arc from $P_t$ to $Q_t$ is strictly convex in the sense that $(x' r'' - r' x'')$ is non-vanishing along the curve.

 \begin{figure}[H]
 \centering
 \includegraphics[height=5cm]{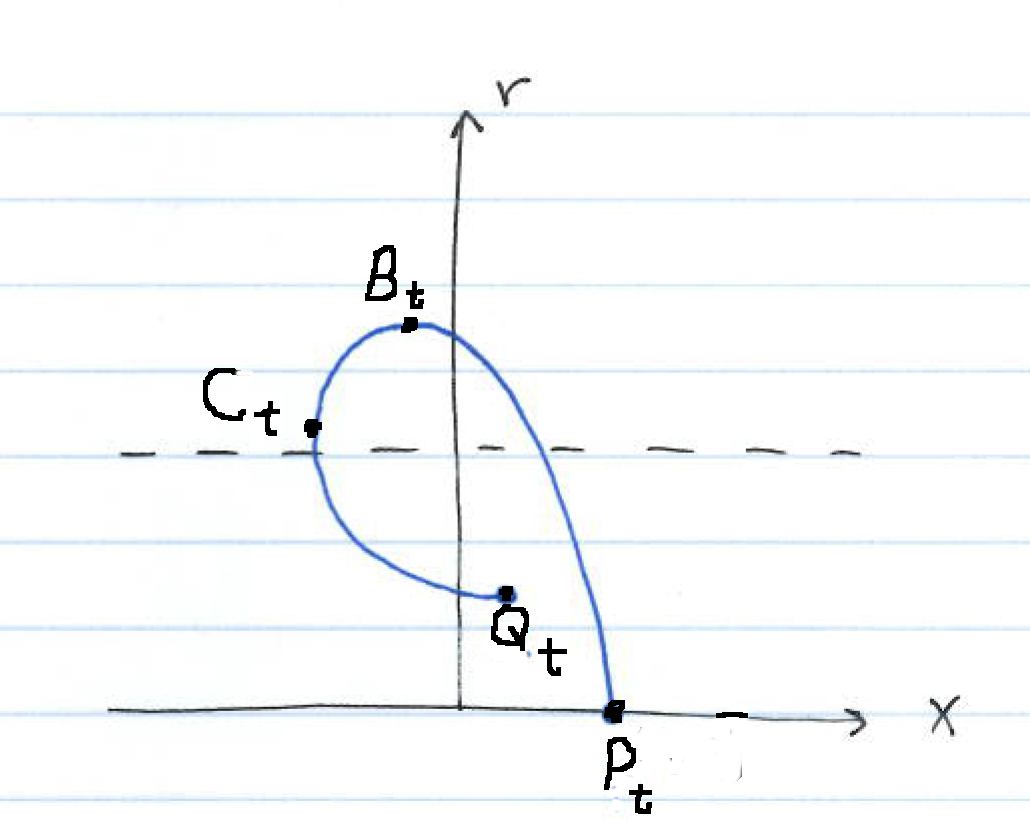}
 \caption{A geodesic arc $S[t]$ with property $\mathscr{P}(t)$.}
 \label{fig:BCQ}
 \end{figure}

To define the initial shooting coordinate $x_*$, we introduce the property $\mathscr{P}(t)$ for the geodesic $S[t]$. As in Figure~\ref{fig:BCQ}, we say that $\mathscr{P}(t)$ holds if
\begin{itemize}
\item[1.] The geodesic $S[t]$ contains points $B_t$, $C_t$, and $Q_t$ as described in Step 1.
\item[2.] The geodesic $S[t]$ is strictly convex as it travels from $P_t$ to $Q_t$. 
\item[3.] The geodesic $S[t]$ crosses the $r$-axis between $P_t$ and $B_t$ and then crosses back over between $C_t$ and $Q_t$.
\end{itemize}
We define the initial shooting coordinate $x_*$ by
\begin{equation}
\label{shoot:value}
x_* = \sup \{ x > 0 \, : \, \mathscr{P}(t) \textrm{ holds for all } t \in (0,x] \}.
\end{equation}
It follows from Step 1 that $x_*$ is well-defined and $x_* \geq \bar{t}$. 

By construction, for $0<t<x_*$, the geodesic $S[t]$ satisfies property $\mathscr{P}(t)$, and by continuous dependence on initial conditions, $S[t]$ converges to $S[x_*]$ as $t \uparrow x_*$. The following result summarizes the main properties of $S[x_*]$ established in~\cite{Drugan 2015}.
\begin{proposition}
Let $x_*$ be the initial shooting value defined in~(\ref{shoot:value}). Then 
\begin{itemize}
\item[1.] $x_* < \sqrt{2n}$.
\item[2.] As $t \uparrow x_*$, the points $B_t$, $C_t$, $Q_t$ on the geodesics $S[t]$ converge to distinct points $B_{x_*}$, $C_{x_*}$, $Q_{x_*}$ on $S[x_*]$ in a compact subset of $\mathbb{H}$.
\item[3.] The geodesic $S[x_*]$ is strictly convex as it travels from $P_{x_*}$ to $Q_{x_*}$.
\item[4.] The geodesic $S[x_*]$ crosses the $r$-axis between $P_{x_*}$ and $B_{x_*}$.
\item[5.] The point $Q_{x_*}$ lies on the $r$-axis.
\end{itemize}
\end{proposition}

It follows from this proposition along with the symmetry of solutions to~(\ref{geo:eqn}) with respect to reflections about the $r$-axis shows that $S[x_*]$ is the profile curve for an immersed $\mathbb{S}^n$ self-shrinker with the shape illustrated in Figure~\ref{fig:is}.

\begin{remark}
In Step 2, the geodesic $S[x_*]$ is analyzed as the limit of the geodesics $S[t]$ as $t \uparrow x_*$. Now, each geodesic $S[t]$ is strictly convex as a curve from $P_t$ to $Q_t$. However, this does not imply that $S[x_*]$ is strictly convex. In fact, without further argument, we can only conclude that $S[x_*]$ is convex in the sense that $(x' r'' - r' x'')$ may vanish but does not change sign. There are indeed examples where the limiting geodesic may not be strictly convex. For instance, the geodesics $S[t]$ converge to the $r$-axis as $t \downarrow 0$. In the case of $S[x_*]$, proving the existence of the point $C_{x_*}$ ensures that $S[x_*]$ is strictly convex as it travels from $P_{x_*}$ to $Q_{x_*}$.
\end{remark}


\subsection{A collection of shooting problems for closed self-shrinkers}
\label{immersed shrinkers}

In this section, we sketch the behavior of geodesics for three shooting problems and illustrate how this behavior can be used to construct closed self-shrinkers. The analysis for the results stated in this section can be found in~\cite{Drugan Kleene 2017}.

\subsubsection{A second immersed sphere self-shrinker}

Consider the shooting problem: $$S[t] \textrm{ is the solution to~(\ref{SSEq}) with } S[t] (0) = (t, 0) \textrm{ and } S[t]' (0) = (0, 1).$$ For this shooting problem we consider the geodesics $S[t]$ when $t>0$ is close to 0. Given $N>0$, there exists $\varepsilon >0$ so that when $0<t<\varepsilon$, the geodesic $S[t]$ is strictly convex as it crosses back and forth over the $r$-axis with $N$ local maximums to the left of the $r$-axis and $N$ local minimums to the right of the $r$-axis.
 
 \begin{figure}[H]
 \centering
 \includegraphics[height=5cm]{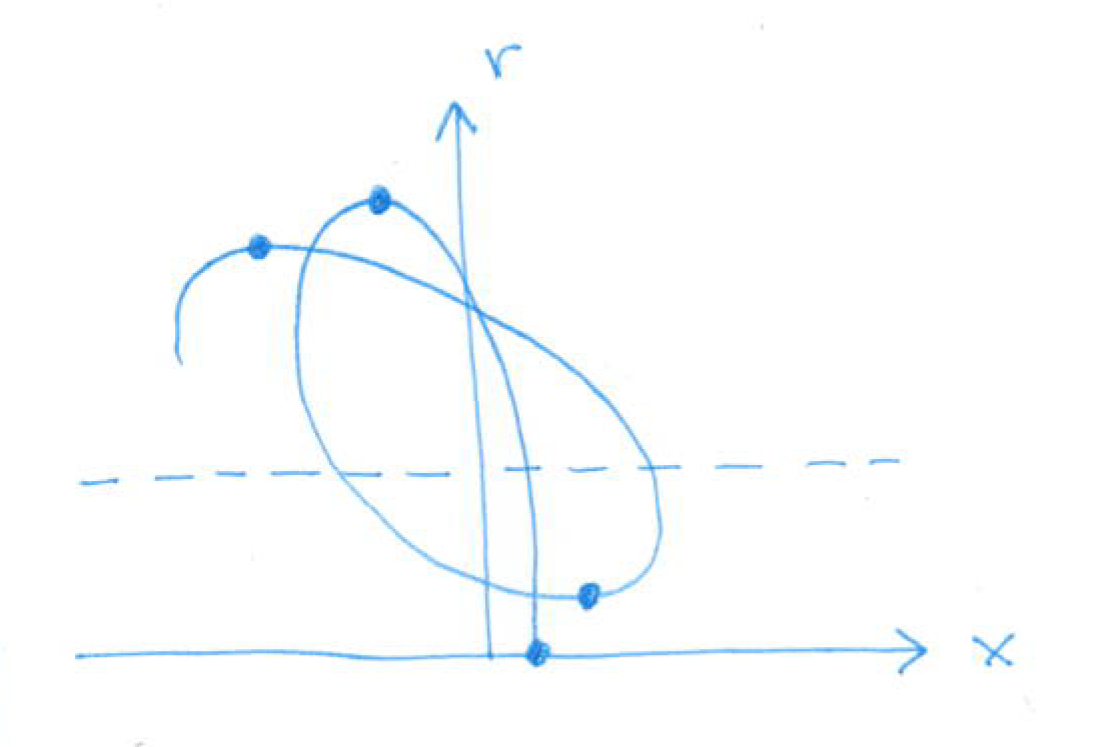}
  \caption{$S[t]$ for $t>0$ close to 0.}
 \end{figure}

As was shown in the previous section, this shooting problem leads to the existence of the profile curve, $S[x_*]$, for an immersed $\mathbb{S}^n$ self-shrinker. We will sketch a proof that there is $x_{**} \in (0, x_*)$ so that $S[x_{**}]$ is the profile curve for a second immersed $\mathbb{S}^n$ self-shrinker.

For $t$ close to $x_*$, the geodesic $S[t]$ has the following shape:

 \begin{figure}[H]
 \centering
 \includegraphics[height=5cm]{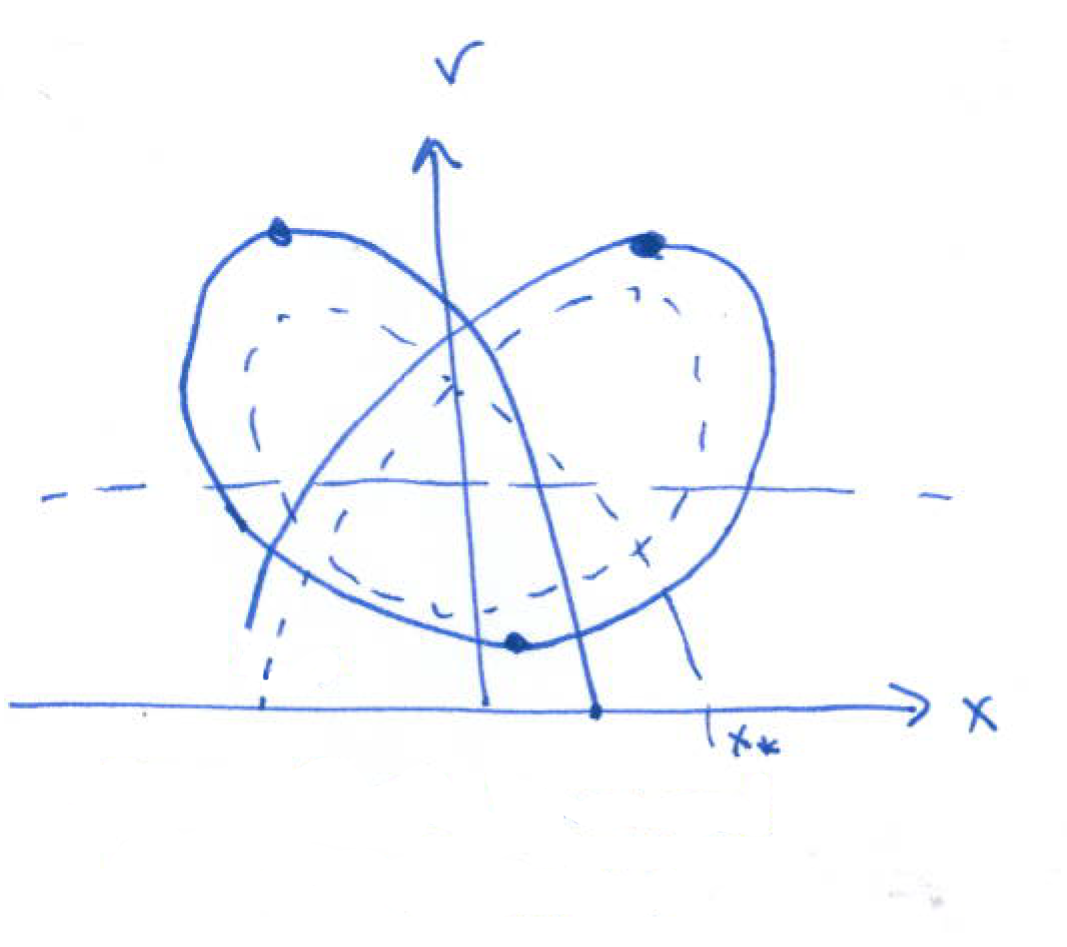}
   \caption{$S[t]$ for $t$ close to $x_*$.}
 \end{figure}

Notice that the second local maximum points lie on different sides of the $r$-axis in the previous two figures. As $t$ varies from $0$ to $x_*$, a continuity argument can be used to show that there is $x_{**} \in (0, x_*)$ so that the second local maximum lies on the $r$-axis. By the symmetry of geodesics with respect to reflections about the $r$-axis, $S[x_{**}]$ intersects the $x$-axis orthogonally at two points and is the profile curve for an immersed $\mathbb{S}^n$ self-shrinker.
 
 \begin{figure}[H]
 \centering
 \includegraphics[height=5cm]{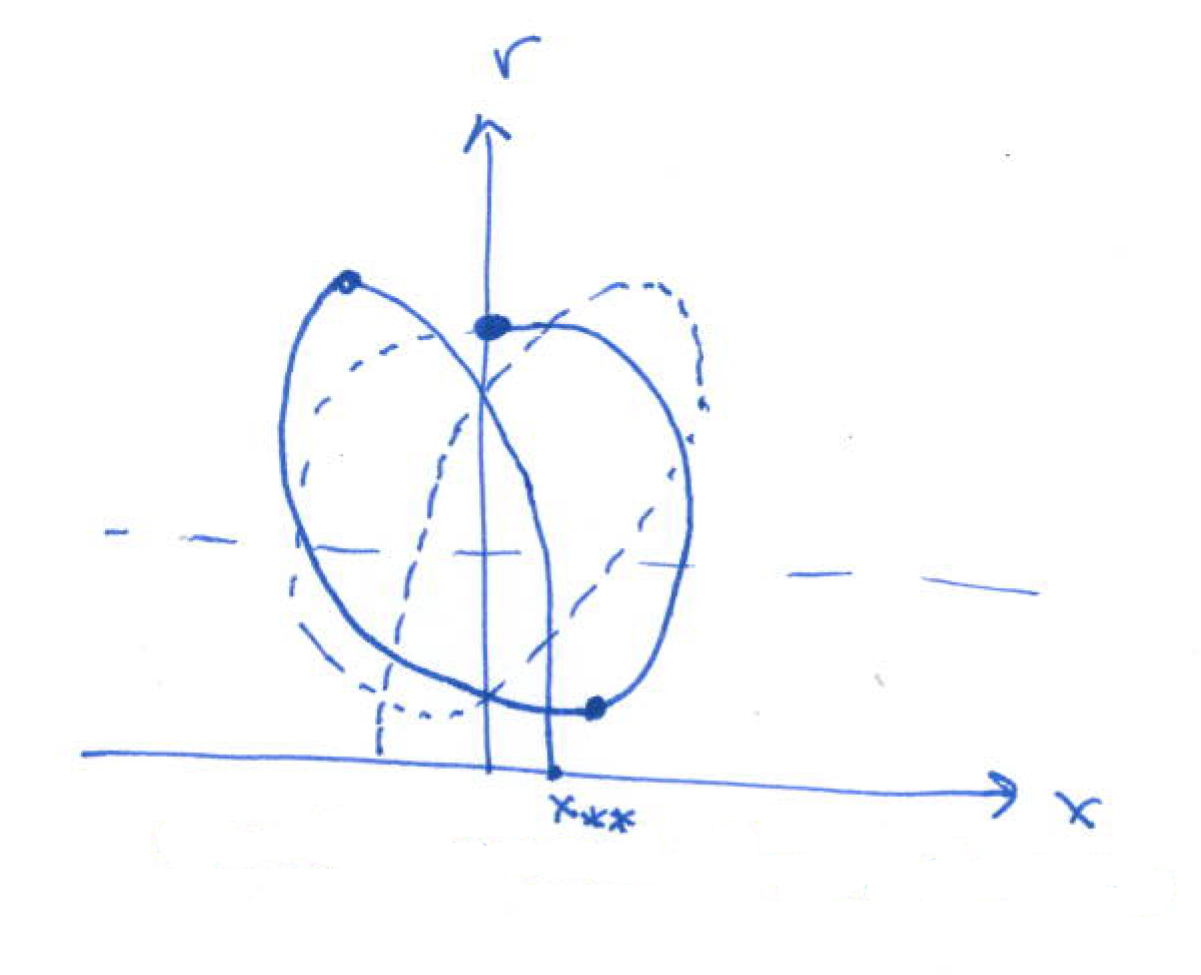}
 \caption{$S[x_{**}]$ is the profile curve for a second immersed $\mathbb{S}^n$ self-shrinker.}
 \end{figure}

\subsubsection{An embedded torus self-shrinker}

Consider the shooting problem: $$T[t] \textrm{ is the solution to~(\ref{SSEq}) with } T[t] (0) = (0,t) \textrm{ and } T[t]' (0) = (1, 0).$$ For this shooting problem we consider the geodesics $T[t]$ when $t>0$ is close to 0. Solutions to this shooting problem behave similarly to solutions from the previous shooting problem, maintaining their strict convexity as they cross back and forth over the $r$-axis with local maximums to the left of the $r$-axis and local minimums (for $s>0$) to the right of the $r$-axis.
 
 \begin{figure}[H]
 \centering
 \includegraphics[height=5cm]{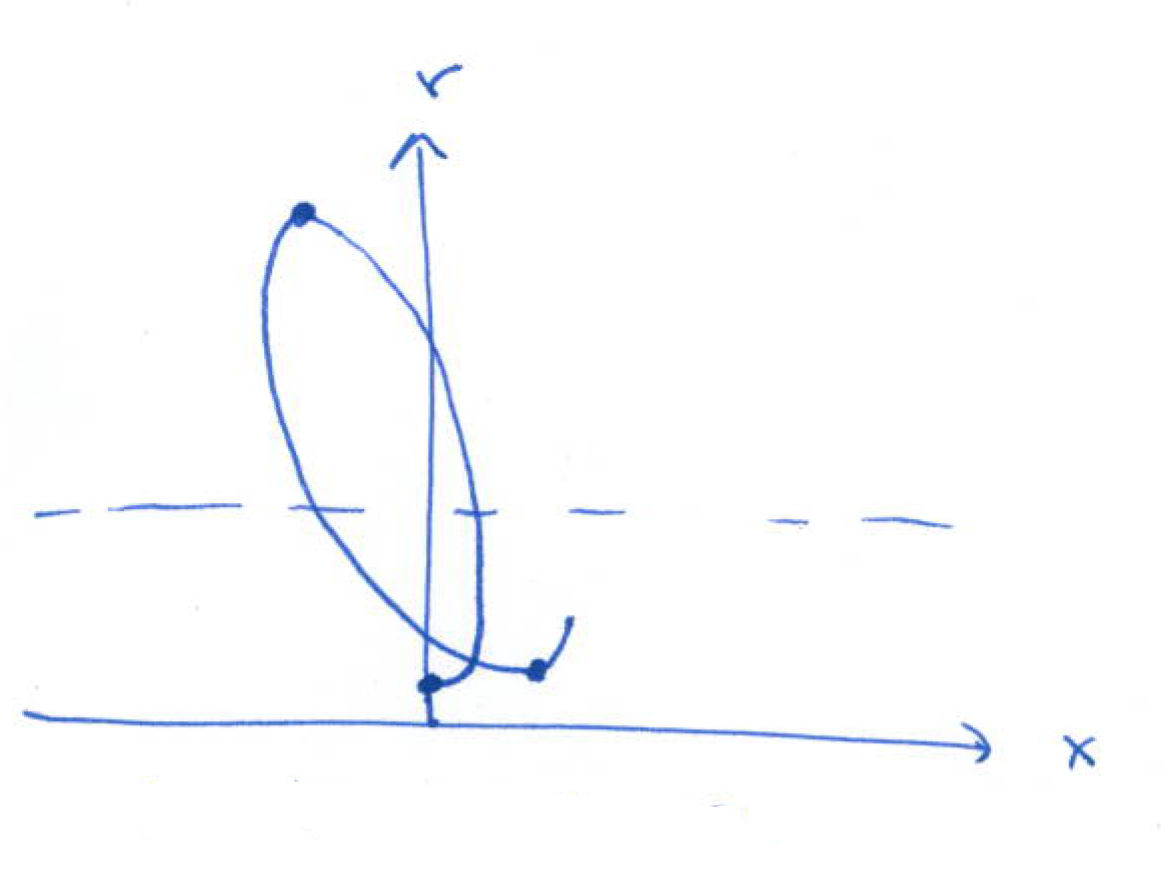}
   \caption{$T[t]$ for $t>0$ close to 0.}
 \end{figure}

Increasing $t$ away from $0$ leads to the profile curve for an immersed $\mathbb{S}^n$ self-shrinker with the shape illustrated in Figure~\ref{fig:is}. In particular, we have two convex geodesic arcs with local maximums on opposite sides of the $r$-axis. A continuity argument can be used to show that there is a simple closed strictly convex geodesic that intersects the $r$-axis orthogonally at two points. Such a geodesic is the profile curve of an $\mathbb{S}^1 \times \mathbb{S}^{n-1}$ self-shrinker.

 \begin{figure}[H]
 \centering
 \includegraphics[height=5cm]{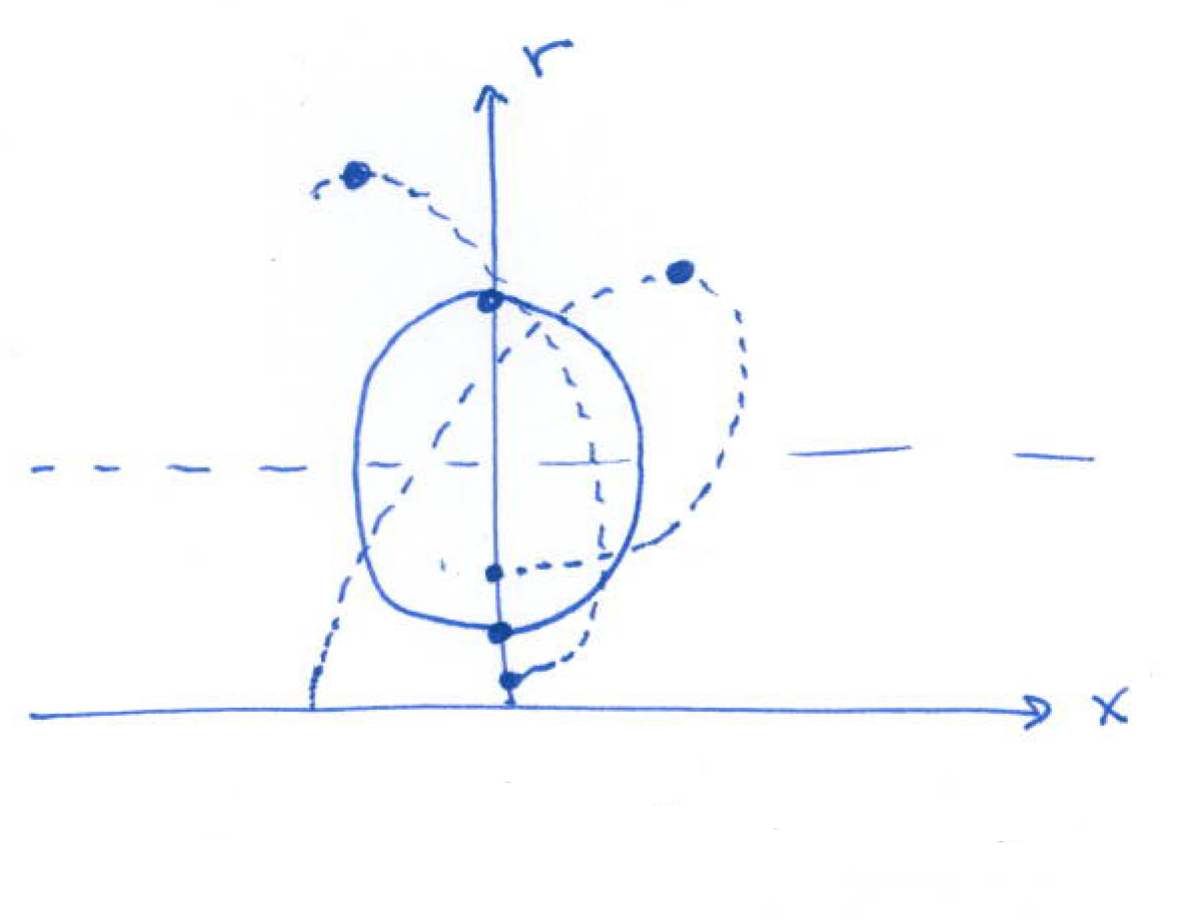}
 \caption{The profile curve of an embedded $\mathbb{S}^1 \times \mathbb{S}^{n-1}$ self-shrinker.}
 \end{figure}

\subsubsection{An immersed torus self-shrinker}

Consider the shooting problem: $$T[t] \textrm{ is the solution to~(\ref{SSEq}) with } T[t] (0) = (0,t) \textrm{ and } S[t]' (0) = (1, 0).$$ For this shooting problem we consider the geodesics $T[t]$ when $t<\sqrt{2(n-1)}$ is close to $\sqrt{2(n-1)}$. Solutions to this shooting problem cross back and forth over the geodesic $r \equiv \sqrt{2(n-1)}$ oscillating in a shape that resembles $\infty$:
 
 \begin{figure}[H]
 \centering
 \includegraphics[height=5cm]{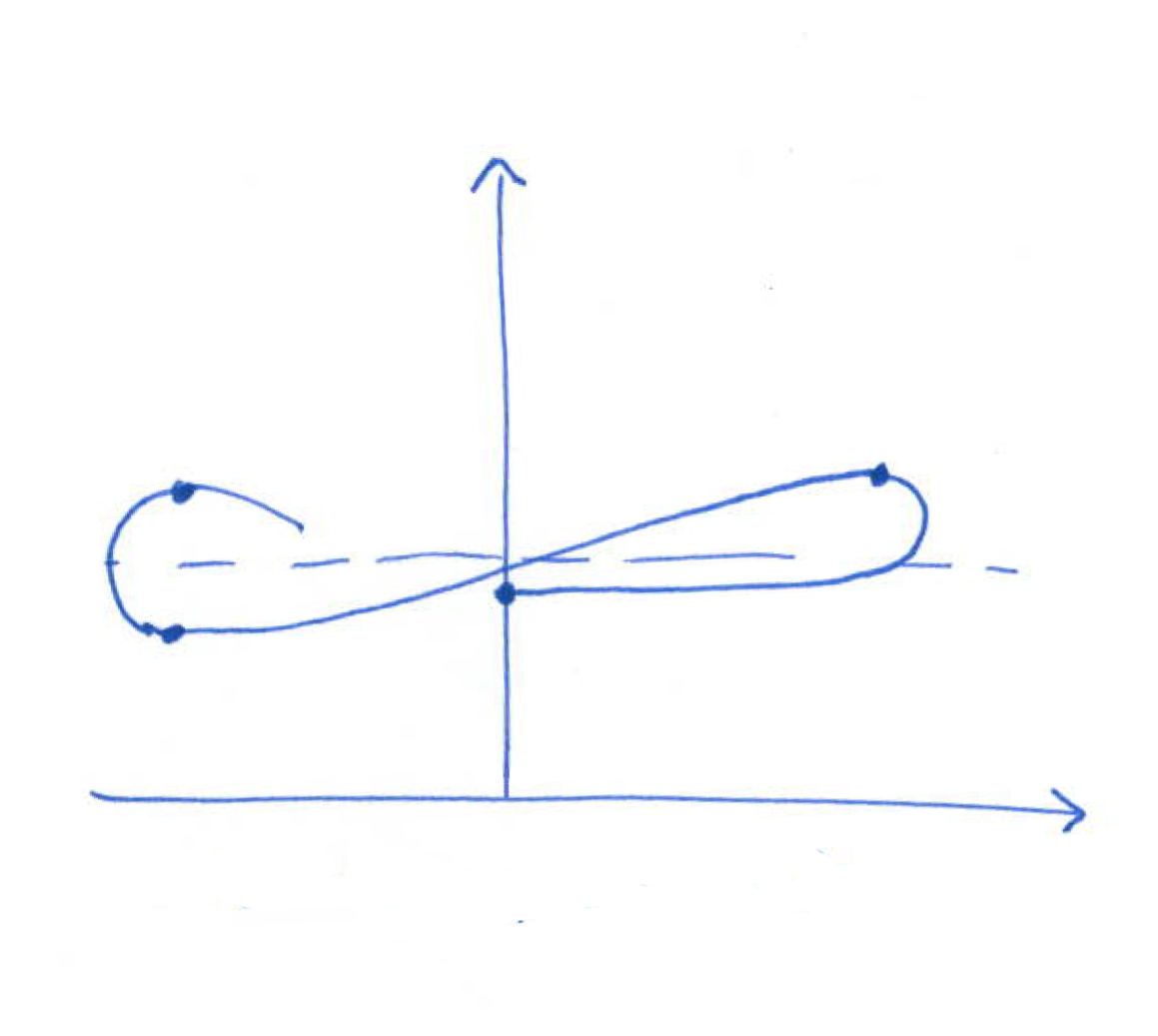}
   \caption{}
 \end{figure}

Decreasing $t$ away from $\sqrt{2(n-1)}$ again leads to the profile curve for an immersed $\mathbb{S}^n$ self-shrinker with the shape illustrated in Figure~\ref{fig:is}. In particular, right before $t$ reaches this profile curve, $T[t]$ has the following shape:

 \begin{figure}[H]
 \centering
 \includegraphics[height=5cm]{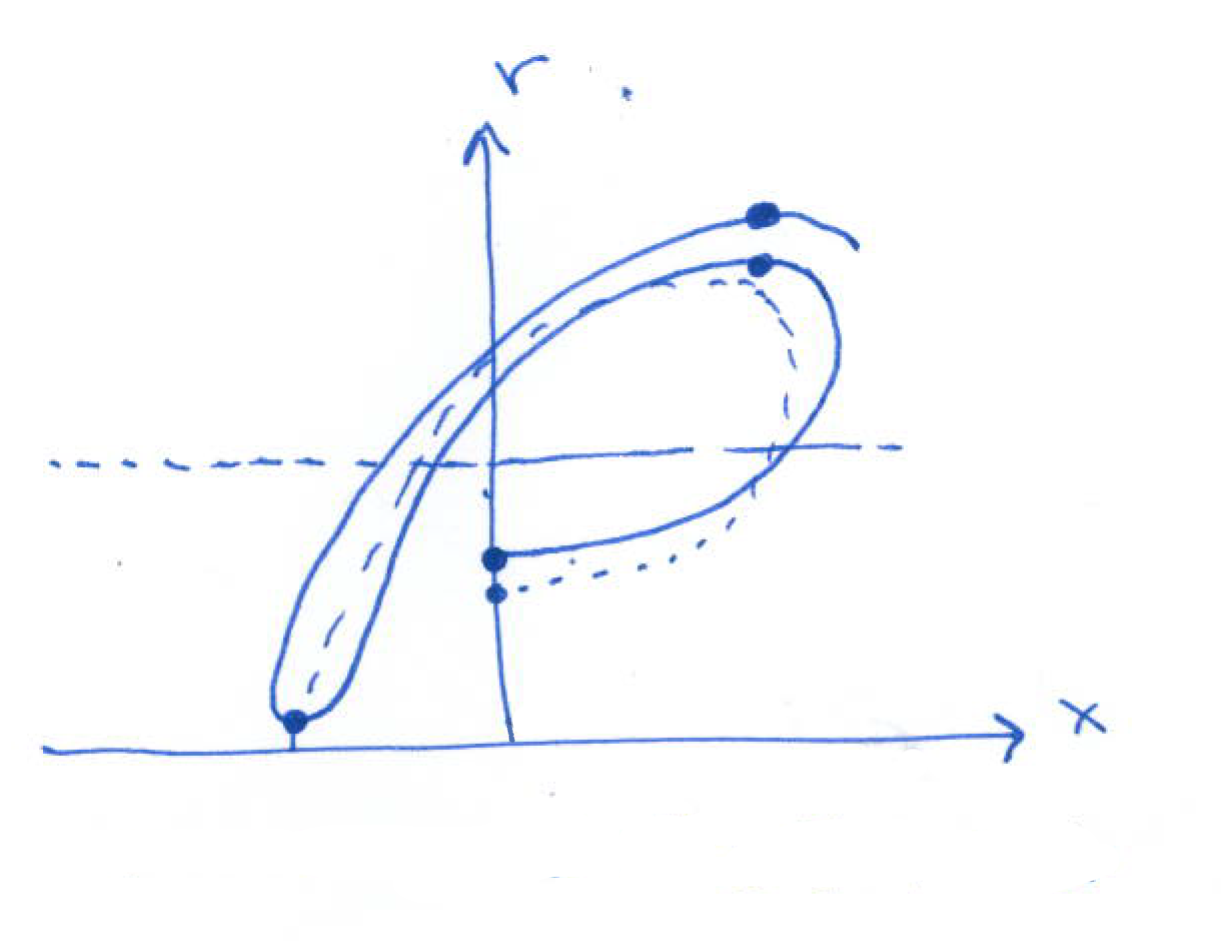}
   \caption{}
 \end{figure}

Comparing the geodesic arcs in the previous two figures, we see that the second local maximum points lie on opposite sides of the $r$-axis. It follows that there is $r_* \in (0, \sqrt{2(n-1)})$ so that $T[r_*]$ intersects the $r$-axis orthogonally at two points: 

 \begin{figure}[H]
 \centering
 \includegraphics[height=5cm]{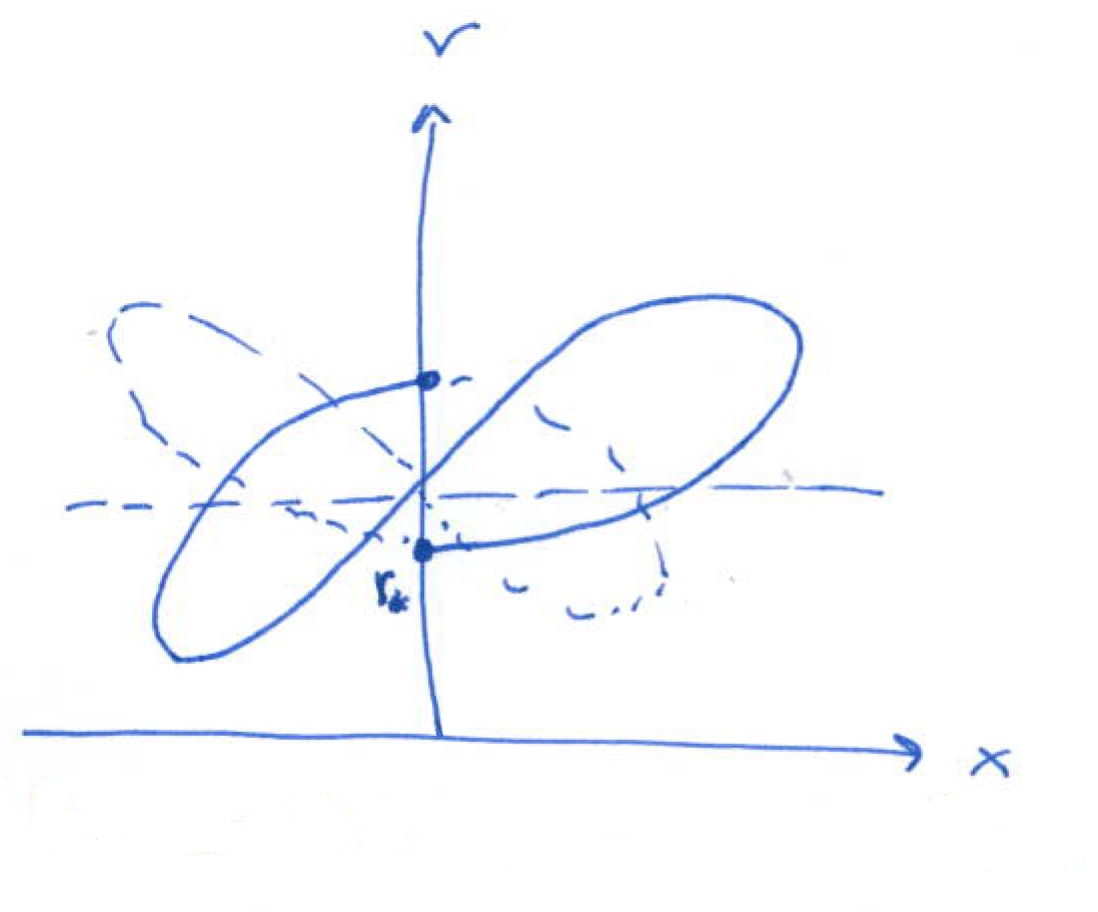}
 \caption{$T[r_*]$ is the profile curve for an immersed $\mathbb{S}^1 \times \mathbb{S}^{n-1}$ self-shrinker.}
 \end{figure}


\subsection{The role of continuity in the shooting method}
\label{Section 5.3}

In this section we discuss the role of continuity in the shooting method for constructing closed geodesics in $(\mathbb{H},g_{Ang})$. First of all, since the shooting method is implemented by varying the initial position and velocity for solutions to (\ref{SSEq}), the continuous dependence of solutions on initial conditions is the main force at work. Next, we observe that the geodesic equation~(\ref{geo:eqn}), re-written as $$\frac{x' r'' - x'' r'}{(x'^2 + r'^2)^{3/2}}  = \left( \frac{n-1}{r}  -\frac{r}{2} \right) \frac{x'}{\sqrt{x'^2 + r'^2}} + \frac{1}{2}x \frac{r'}{\sqrt{x'^2 + r'^2}},$$ gives uniform bounds for the Euclidean curvature of geodesic arcs in a fixed compact subset of $\mathbb{H}$. This tells us that as we vary the initial conditions in the shooting problem, the limit of geodesic arcs in a fixed compact subset of $\mathbb{H}$ will not develop corners or collapse onto a curve with multiplicity. (We note that this type of behavior does occur at the boundary of $\mathbb{H}$ when geodesics converge to half-entire graphs with multiplicity~\cite{Drugan Kleene 2017}). 

Now, the geodesic equation~(\ref{geo:eqn}) is symmetric with respect to reflections across the $r$-axis, so the existence of a closed geodesic can be established by finding a geodesic arc that intersects the $r$-axis orthogonally at two points. One way to find such a geodesic arc is to build a framework where the following three properties hold:

\begin{enumerate}

\item[1.] There is a point $P_1$ so that the geodesic arc obtained from shooting orthogonally to the $r$-axis at $P_1$ with tangent angle 0 contains a point $Q_1$, located to the left of the $r$-axis with tangent angle $\pi$.

 \begin{figure}[H]
 \centering
 \includegraphics[height=4cm]{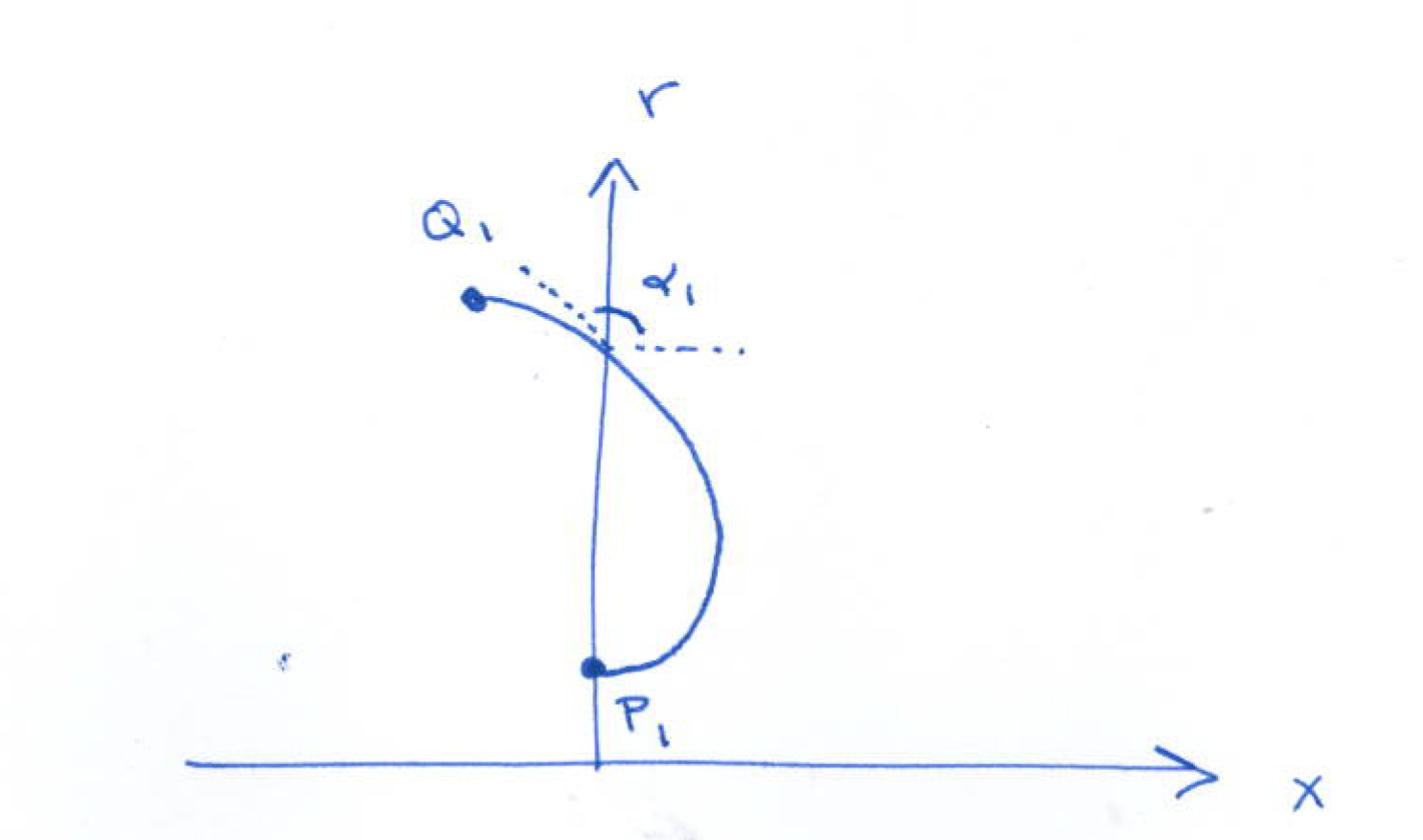}
    \caption{}
 \end{figure}

\item[2.] There is a point $P_2$ so that the geodesic arc obtained from shooting orthogonally to the $r$-axis at $P_2$ with tangent angle 0 contains a point $Q_2$, located to the right of the $r$-axis with tangent angle $\pi$.

 \begin{figure}[H]
 \centering
 \includegraphics[height=4cm]{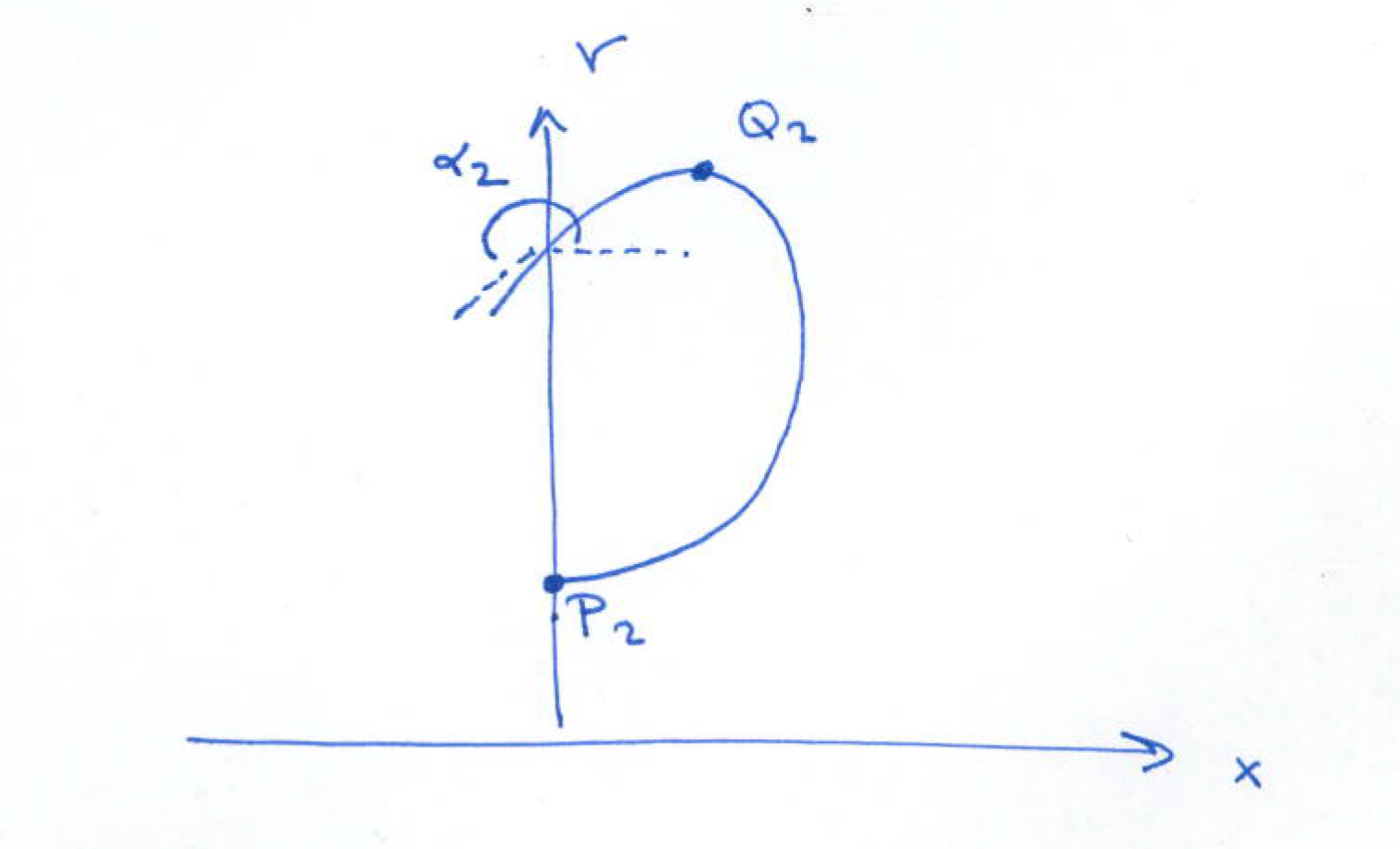}
    \caption{}
 \end{figure}

\item[3.] For each point $P$ on the $r$-axis, between $P_1$ and $P_2$, the geodesic arc obtained from shooting orthogonally to the $r$-axis at $P$ with tangent angle 0 travels from $P$ to a point $Q$ with tangent angle $\pi$. In addition, as $P$ varies from $P_1$ to $P_2$, the geodesic arcs connecting $P$ to $Q$ vary smoothly and remain inside a compact subset of $\mathbb{H}$. (We note that under these assumptions, (\ref{geo:eqn}) implies that the geodesics are strictly convex at $Q$).

\end{enumerate}

\begin{remark}
\label{rm:Q}
In this example, continuity, combined with the behavior of the geodesic arcs through $P_1$ and $P_2$, guarantees that there is a point $P_*$ on the $r$-axis (between $P_1$ and $P_2$) so that the geodesic arc from $P_*$ to $Q_*$ intersects the $r$-axis orthogonally at $Q_*$. Notice how the third property is stated to ensure that continuity can be applied to the problem. For a shooting problem like this, the existence of the points $Q$ needs to be established as continuous dependence on initial conditions does not guarantee that the orthogonal tangent lines (tangent angle $\pi$) will be preserved as $P$ varies. 
\end{remark}

\begin{remark}
\label{rm:alt}
As an alternative to tracking the points $Q$ as $P$ varies from $P_1$ to $P_2$, the crossing angles $\alpha$, where the geodesic arc meets the $r$-axis could be studied instead. In the above illustrations, $\alpha_1$ is less than $\pi$ and $\alpha_2$ is greater than $\pi$. In this case, if the crossing points exist, the crossing angles vary continuously, and the geodesic arcs remain in a compact subset of the domain, then there is a geodesic arc that intersects the $r$-axis with a crossing angle of $\pi$.
\end{remark}

 
\section{Self-shrinkers with bi-rotational symmetry}    \label{Section 6} 

In this section, we consider the problem of constructing closed self-shrinkers with bi-rotational symmetry. A hypersurface in $\mathbb{R}^{M_{1}+M_{2}+2}$ has \emph{bi-rotational symmetry} if it can be written in the form $$\mathbf{X}(s,p_1,p_1) = \left( x(s)  {\mathbf{p}_{1}} , y(s)  {\mathbf{p}_{2}}   \right),$$ where $(x(s),y(s))$ is a curve in the positive quadrant $\mathcal{Q} = \{(x,y) \in \mathbb{R}^2 \, | \, x>0, y>0 \}$, ${\mathbf{p}_{1}}  \in {\mathbb{S}}^{M_{1}}$, ${\mathbf{p}_{2}}  \in {\mathbb{S}}^{M_{2}}$, and  $M_1$ and $M_2$ are positive integers.

Historically, bi-rotational symmetry has produced rich examples of solutions to classical problems. Hsiang \cite{Hsiang 1982} proved the existence of infinitely many distinct bi-rotational CMC immersions of ${\mathbb{S}}^{N-1}$ in ${\mathbb{R}}^{N \geq 4}$. These examples show that Hopf's Theorem does not extend to higher dimensions. Alencar-Barros-Palmas-Reyes-Santo \cite[Theorem 1.2]{Alencar Barros Palmas Reyes Santos 2005} proved the existence of various bi-rotational minimal hypersurfaces in ${\mathbb{R}}^{M_{1}+M_{2}+2 \geq 8}$ with $M_{1}, M_{2} \geq 2$, that are asymptotic or doubly asymptotic to the minimal Clifford cone. Bombieri-De Giorgi-Giusti \cite[Section IV]{Bombieri De Giorgi Giusti 1969} investigated \emph{bi-radial} graphs of the form $x_{2m+1} = F (  \sqrt{ {x_{1}}^{2} + \cdots + {x_{m}}^{2} \, }, \sqrt{{x_{m+1}}^{2} + \cdots + {x_{2m}}^{2}} )$. They proved the existence of an entire, non-flat minimal graph in $\mathbb{R}^9$, which showed that Bernstein's Theorem does not hold in $\mathbb{R}^{N \geq 9}$. Recently, Del Pino-Kowalczyk-Wei \cite{del Pino Kowalczyk Wei 2011} studied the asymptotic behavior of the Bombieri-De Giorgi-Giusti graph and proved the existence of a counterexample to De Giorgi's conjecture for the Allen-Cahn equation in $\mathbb{R}^{N \geq 9}$.

The following reduction tells us that the profile curve of a bi-rotational self-shrinker is a weighted geodesic in $\mathcal{Q}$. 
 
\begin{proposition}[\textbf{Bi-rotational self-shrinkers from weighted geodesics in $\mathcal{Q}$}] \label{MCV equality} 
Let $\alpha<0$ be a constant, let $M_1$ and $M_2$ be positive integers, and let  $\gamma(s) = (x(s), y(s))$, $s \in I$, be an immersed curve in the positive quadrant $\mathcal{Q}$. The following conditions are equivalent. 
\begin{itemize}

\item[1.] The bi-rotational hypersurface
\[
{\Sigma}^{M_{1}+M_{2}+1}   = \left\{  \; \mathbf{X} = ( x(s)  {\mathbf{p}_{1}} , y(s) {\mathbf{p}_{2}}   ) \in {\mathbb{R}}^{M_{1}+M_{2}+2} \;\; \vert \; \;  s \in I,  \; {\mathbf{p}_{1}}  \in {\mathbb{S}}^{M_{1}}, \; {\mathbf{p}_{2}}  \in {\mathbb{S}}^{M_{2}} \; \right\}
\]
is a self-shrinker, satisfying ${\Delta}_{ {\mathbf{g}}_{{}_{\Sigma}}  }  \mathbf{X} = \alpha \mathbf{X}^{\bot}$.

\item[2.] The profile curve $\gamma(s) = (x(s),y(s))$ is a weighted geodesic in the positive quadrant $\mathcal{Q}$ equipped with the density 
\[
 x^{M_1} y^{M_2} e^{\, \alpha \frac{x^2 + y^2}{2}}.
\]

\item[3.] The profile curve  $\gamma(s) = (x(s), y(s))$ satisfies the geodesic equation 
\begin{equation}
\label{bgeo:eqn}
\frac{x' y'' - x'' y'}{x'^2 + y'^2} = - \left( \frac{M_1}{x}  + \alpha x \right) y' + \left( \frac{M_2}{y}  + \alpha y \right) x'.
\end{equation}
\end{itemize}
\end{proposition} 

As in the rotational symmetry case, the shooting method can be used to explore profile curves of bi-rotational self-shrinkers, and the existence of a closed self-shrinker with bi-rotational symmetry is equivalent to finding a solution to~(\ref{bgeo:eqn}) that is closed or intersects the boundary of $\mathcal{Q}$ orthogonally at two points.

\subsection{The shooting method for bi-rotational self-shrinkers}

In this section we adopt the normalization $\alpha= - \frac{1}{2}$, and we make the additional assumption that $ M_1 = M_2 = M \geq 1$. Then, the geodesic equation (\ref{bgeo:eqn}) becomes
\begin{equation}
\label{bgeo:eqn2}
\frac{x' y'' - x'' y'}{x'^2 + y'^2} = - \left( \frac{M}{x}  -\frac{1}{2}x \right) y' + \left( \frac{M}{y}  -\frac{1}{2} y \right) x'.
\end{equation}
Notice that for a geodesic $y=y(x)$ written as a graph over the $x$-axis, we have the following equation which resembles equation~(\ref{ode:graph}) for self-shrinkers with rotational symmetry:
\begin{equation}
\label{bgeo:eqn3}
\frac{y''}{1 + y'^2} =  - \left( \frac{M}{x} - \frac{1}{2}x \right) y' + \left( \frac{M}{y}  -\frac{1}{2} y \right).
\end{equation}

Examples of bi-rotational self-shrinkers in ${\mathbb{R}}^{2M+2}$ include the minimal Clifford cone, round cylinders $\mathbb{S}^{M} \times \mathbb{R}^{M+1}$ of radius $\sqrt{2M \,}$ with `axis' through the origin, and the round sphere of radius $\sqrt{2 (2M + 1)}$ centered at the origin. We have the following profile curves for these examples:
\begin{itemize}
\item[1.] \emph{Minimal Clifford cone}:  $y = x$, 
\item[2.] \emph{Round cylinders}: $y = \sqrt{2 M\,}$ and $x = \sqrt{2M\,}$,
\item[3.] \emph{Round sphere}: $x^2 + y^2 = 2 (2M + 1)$.
\end{itemize}

 \begin{figure}[H]
 \centering
 \includegraphics[height=4.785cm]{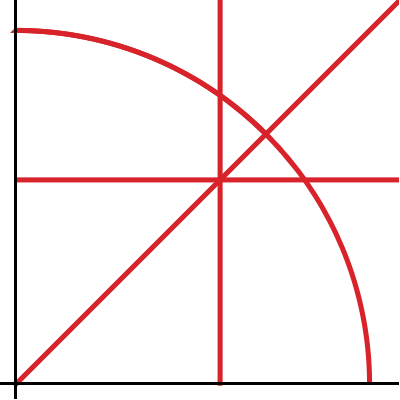}
   \caption{Geodesics corresponding to the Clifford cone, two round cylinders, and the round sphere.}
 \end{figure}

The assumption $M_1 = M_2 = M$ introduces an additional symmetry as the geodesic equation (\ref{bgeo:eqn2}) is now symmetric with respect to reflections about the line $y=x$. Consequently, a closed geodesic can be constructed by finding a geodesic arc that intersects the line $y=x$ orthogonally at two points.

Reparametrizing a solution to (\ref{bgeo:eqn2}) so that the curve satisfies $x'(s)^2 + y'(s)^2 = 1$, shows that the tangent angle $\alpha(s)$ is a solution to the system
\begin{equation} 
\label{bi:SSEq}
\left\{
\begin{array}{lll}
x'(s) & = & \cos \alpha(s), \\
y'(s) & = & \sin \alpha(s), \\
\alpha'(s) & = & - \left( \frac{M_1}{x(s)} - \frac{x(s)}{2} \right) \sin \alpha(s) + \left( \frac{M_2}{y(s)} - \frac{y(s)}{2} \right) \cos \alpha(s).
\end{array}
\right.
\end{equation}
This leads to the shooting problem: 
\begin{equation} 
\label{bi:shoot}
T[t] \textrm{ is the solution to~(\ref{bi:SSEq}) with } T[t] (0) = (t,t) \textrm{ and } T[t]' (0) = \left(-\frac{1}{\sqrt{2}}, \frac{1}{\sqrt{2}}\right).
\end{equation}
As described in Section~\ref{Section 5.3}, the existence of a closed geodesic can be established by showing that the above shooting problem exhibits the two types of behavior illustrated in Figure~\ref{bishoot}.

 \begin{figure}[H]
 \centering
 \includegraphics[height=4.785cm]{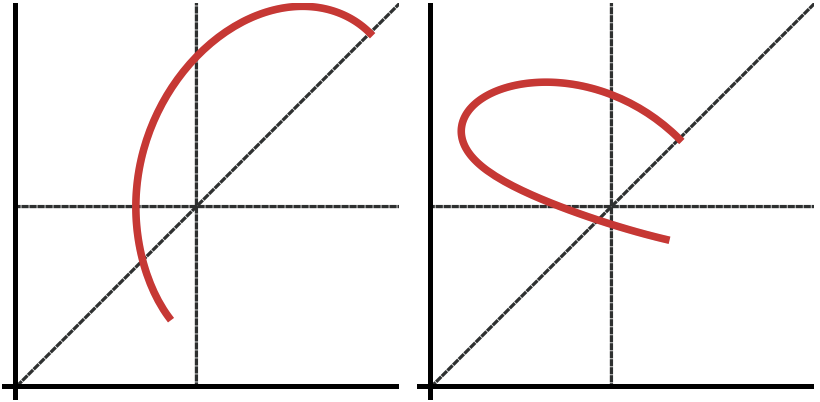}
   \caption{}
   \label{bishoot}
 \end{figure}

McGrath has posted a preprint~\cite{McGrath 2015} which shows that $T[t]$ exhibits the behavior illustrated on the left of Figure~\ref{bishoot}, for large values of the parameter $t$. In this preprint, McGrath presents a shooting method argument similar to the ones used to construct closed self-shrinkers with rotational symmetry in~\cite{Angenent 1992} and~\cite{Drugan 2015}. The approach is to analyze the geodesics $T[t]$ as $t$ decreases from $\infty$ and deduce through continuity that $T[t_*]$ is a closed geodesic for some $t_*>0$. We note that the presence of the additional term in (\ref{bgeo:eqn2}), when compared to (\ref{geo:eqn}), makes the analysis of bi-rotational geodesics more complicated than their rotational counterparts.

\subsection{Numerical approximations for profile curves of bi-rotational self-shrinkers in ${\mathbb{R}}^{4}$}
\label{birot examples}

In this section we present numerical approximations of \emph{symmetric} profile curves for closed self-shrinkers with bi-rotational symmetry in the case where $\alpha = - \frac{1}{2}$ and $M_1 = M_2 = 1$. We used Wolfram Mathematica to plot numerical solutions to the system (\ref{bi:SSEq}) for various initial values in the shooting problem (\ref{bi:shoot}). 

\begin{enumerate}

\item[1.] \emph{Embedded ${\mathbb{T}}^{3}$ self-shrinker in ${\mathbb{R}}^{4}$}: A detailed analysis of (\ref{bgeo:eqn3}), adapted from the crossing arguments in \cite{Drugan 2015} and \cite{Drugan Kleene 2017}, confirms there is a $t_2 > 0$ so that the geodesic $T[t]$ has the initial shape illustrated on the left of Figure~\ref{bishoot} for $t \geq t_2$. Numerics show that there is a $t_1 > \sqrt{6}$ so that $T[t_1]$ has the following initial shape:

 \begin{figure}[H]  
 \centering
 \includegraphics[height=4.785cm]{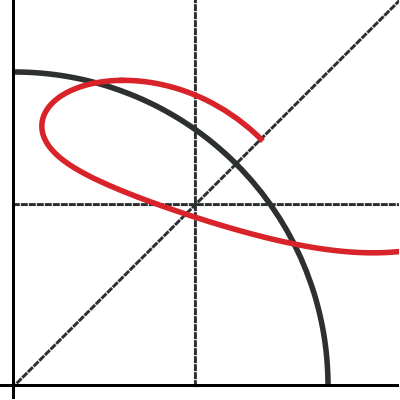}
   \caption{}  \label{ray lower} 
 \end{figure}

The framework of Remark~\ref{rm:alt} can now be implemented to prove the existence of a simple, closed geodesic. However, as discussed in Section~\ref{Section 5.3}, additional details on the behavior of the geodesics $T[t]$, $t \in [t_1,t_2]$ are needed to run the continuity argument, and the existence of a geodesic with the shape of $T[t_1]$ in Figure~\ref{ray lower} still needs to be established. Here is a numerical approximation of a simple, closed geodesic:

 \begin{figure}[H]  
 \centering
 \includegraphics[height=4.785cm]{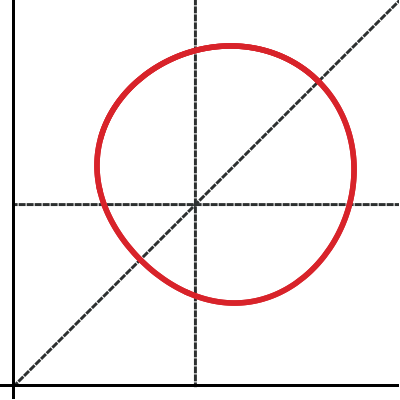}
   \caption{}  \label{embedded T3} 
 \end{figure}

\bigskip

\item[2.] \emph{Three immersed ${\mathbb{T}}^{3}$ self-shrinkers in ${\mathbb{R}}^{4}$}:  

 \begin{figure}[H]  
 \centering
 \includegraphics[height=4.785cm]{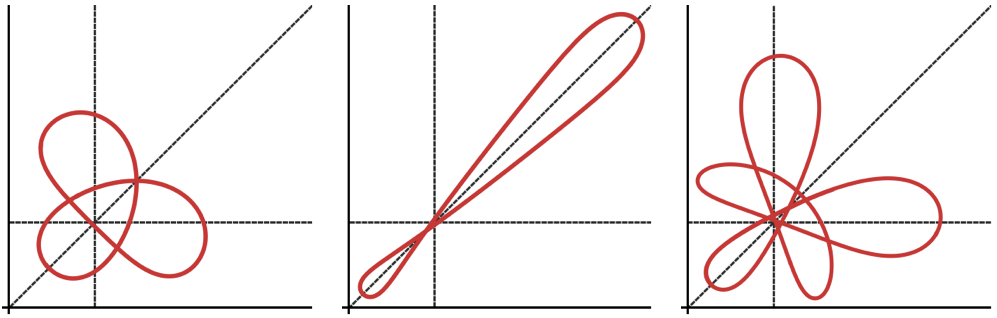}
   \caption{}  \label{immersed T3} 
 \end{figure}

\bigskip

\item [3.] \emph{Two immersed ${\mathbb{S}}^{3}$ self-shrinkers in ${\mathbb{R}}^{4}$}: 

 \begin{figure}[H]     
 \centering
 \includegraphics[height=4.785cm]{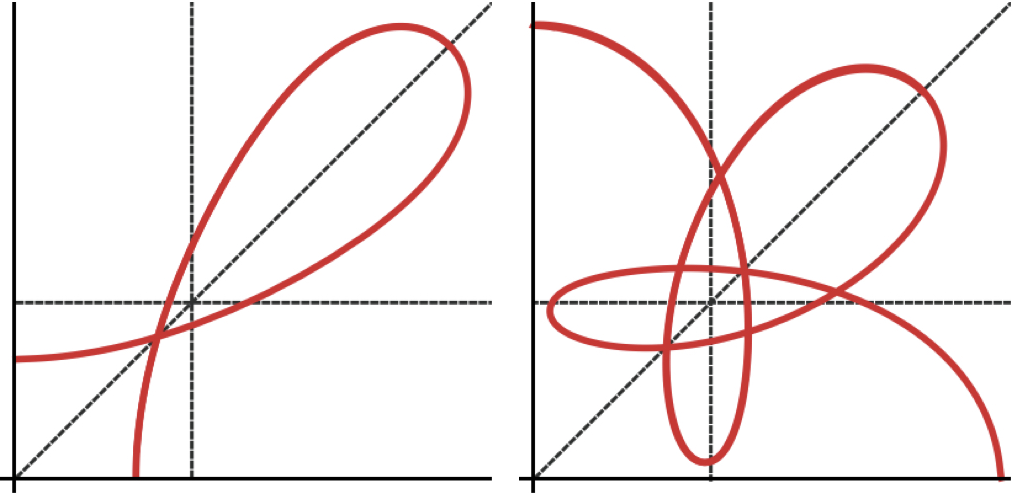}
   \caption{}  \label{immersed S3}
 \end{figure}

\end{enumerate}


 \bigskip
 
\section{Open Problems}  \label{Section 7} 

We end the survey with a list of open problems for closed self-shrinkers. One may also raise similar questions for the existence and uniqueness of closed self-shrinkers in higher dimensions.


\bigskip

\begin{problem} 
Is Angenet's rotational torus the only closed, embedded, genus 1 self-shrinker in $\mathbb{R}^3$?
\end{problem}

The uniqueness of Angenent's torus is even unknown in the class of self-shrinkers with rotational symmetry. The embeddedness assumption is necessary due to the immersed examples constructed in~\cite{Drugan Kleene 2017}. Very recently, motivated by Lawson's Theorem \cite{Lawson 1970} for embedded minimal surfaces in the three-dimensional sphere, Mramor and Wang \cite[Corollary 1.2]{Mramor Wang 2017} exploited the variational characterization of self-shrinkers (Section \ref{variational eqns}) to show that an embedded self shrinking 
torus in $\mathbb{R}^3$ is un-knotted. 


\bigskip

\begin{problem}
Is the round sphere centered at the origin the only embedded $\mathbb{S}^{3}$ self-shrinker in $\mathbb{R}^{4}$?
\end{problem}

The embeddedness assumption is necessary due to the existence of immersed and non-embedded $\mathbb{S}^3$ self-shinkers constructed in~\cite{Drugan 2015, Drugan Kleene 2017}. Uniqueness is known in the rotational case~\cite{Drugan thesis, Kleene Moller 2014}.


\bigskip

\begin{problem}
Existence and uniqueness of closed self-shrinkers with bi-rotational symmetry.
\end{problem}

Are there bi-rotational self-shrinkers for the numeric profile curves presented in Section~\ref{birot examples}? Are there any uniqueness results for these or other bi-rotational examples?
 
  \bigskip
  \bigskip
  \bigskip
     


\end{document}